\documentclass[preprint]{imsart}

\RequirePackage{amsthm, amsmath, amsfonts, amssymb}
\RequirePackage{amsopn, nicefrac, mathtools, dsfont}
\RequirePackage{thm-restate}
\RequirePackage{breqn}
\RequirePackage{stmaryrd}
\RequirePackage{xfrac}
\RequirePackage[numbers]{natbib}
\RequirePackage[colorlinks,citecolor=blue,urlcolor=blue]{hyperref}
\RequirePackage{graphicx}
\RequirePackage{enumitem}
\setlist[enumerate, 1]{label=(\roman*)}
\usepackage[ruled,algo2e]{algorithm2e}
\usepackage{algorithmic}

\RequirePackage[english]{babel}
\RequirePackage[T1]{fontenc}
\RequirePackage[utf8]{inputenc}
\RequirePackage{xspace}

\arxiv{2010.00000}
\startlocaldefs
\theoremstyle{plain}
\newtheorem{hyp}{Assumption}

\newtheorem{thm}{Theorem}[section]
\newtheorem{lemma}[thm]{Lemma}

\theoremstyle{definition}

\theoremstyle{remark}

\newtheorem{remark}{Remark}

\newcommand*{\sepline}{%
  \begin{center}
    \rule[1ex]{6cm}{.5pt}
  \end{center}
}

\DeclareMathOperator*{\argmin}{arg\,min\,}

\renewcommand{\mid}{\;\ifnum\currentgrouptype=16 \middle\fi|\;}

\newcommand*{\bv}{\mathbf{v}}
\newcommand*{\bV}{\mathbf{V}}
\newcommand*{\bw}{\mathbf{w}}
\newcommand*{\bW}{\mathbf{W}}
\newcommand*{\bx}{\mathbf{x}}

\newcommand*{\bY}{\mathbf{Y}}
\newcommand*{\bZ}{\mathbf{Z}}
\newcommand*{\mle}{\text{\tiny MLE}}
\newcommand*{\vem}{\text{\tiny VEM}}
\newcommand*{\btheta}{\boldsymbol{\theta}}
\newcommand*{\bmu}{\boldsymbol{\mu}}
\newcommand*{\bSigma}{\boldsymbol{\Sigma}}
\newcommand*{\bB}{\mathbf{B}}
\newcommand*{\bC}{\mathbf{C}}
\newcommand*{\bo}{\mathbf{o}}

\newcommand*{\bm}{\mathbf{m}}
\newcommand*{\bS}{\mathbf{S}}
\newcommand*{\BB}[1]{#1}
\makeatletter
\newcommand*{\cp}[1]{\vcenter{\hbox{\scalebox{1.39}{\ensuremath#1}}}}
\makeatother

	{\begin{list}{}{%
   		\settowidth{\labelwidth}{#1}%
   		\setlength{\leftmargin}{\labelwidth}\addtolength{\leftmargin}{\labelsep}}}%
  	{\end{list}}

\endlocaldefs
\mathtoolsset{showonlyrefs}
\begin{document}
\begin{frontmatter}

    \title{Importance sampling-based gradient method for dimension reduction in Poisson log-normal model}
    \runtitle{Importance sampling based gradient method for PLN-PCA}

    \begin{aug}
        \author[MIAPS]{\fnms{Bastien}~\snm{Batardière}\ead[label=e1]{bastien.batardiere@inrae.fr}\orcid{0009-0001-3960-7120}}
        \author[MIAPS]{\fnms{Julien}~\snm{Chiquet}\ead[label=e2]{julien.chiquet@inrae.fr}\orcid{0000-0002-3629-3429}},
        \author[MIAPS]{\fnms{Joon}~\snm{Kwon}\ead[label=e3]{joon.kwon@inrae.fr}\orcid{0000-0002-3464-9081}},
        \and
        \author[CEREMADE, MIAPS]{\fnms{Julien}~\snm{Stoehr}\ead[label=e4]{stoehr@ceremade.dauphine.fr}\orcid{0000-0002-7813-0185}}
        \address[MIAPS]{Université Paris-Saclay, AgroParisTech, INRAE, UMR MIA Paris-Saclay, 91120, Palaiseau, France\printead[presep={,\ }]{e1,e2,e3}}
        \address[CEREMADE]{Universit\'e Paris-Dauphine, Universit\'e PSL, CNRS, CEREMADE, 75016 Paris, France\printead[presep={,\ }]{e4}}
    \end{aug}

    \begin{abstract}
    High-dimensional count data poses significant challenges for statistical
    analysis, necessitating effective methods that also preserve
    explainability. We focus on a low rank constrained variant of the Poisson
    log-normal model, which relates the observed data to a latent low-dimensional
    multivariate Gaussian variable via a Poisson distribution. Variational
    inference methods have become a golden standard solution to infer such a
    model. While computationally efficient, they usually lack theoretical
    statistical properties with respect to the model. To address this issue we
    propose a projected stochastic gradient scheme that directly maximizes the
    log-likelihood. We prove the convergence of the proposed method when using
    importance sampling for estimating the gradient.
    Specifically, we achieve a convergence rate of
    \BB{$O(T^{\nicefrac{-1}{2}} + N^{-1})$,
     % or $O(N^{\nicefrac{1}{2}}T^{\nicefrac{-1}{2}} + N^{-1})$
     % when accounting for sampling complexity,
    }
      where $T$ denotes the
     number of iterations and $N$ represents the number of Monte Carlo samples. The latter
    follows from a novel descent lemma for non convex $L$-smooth objective
    functions, and random biased gradient estimate. We also demonstrate
    numerically the efficiency of our solution compared to its variational
    competitor. Our method not only scales with respect to the number of
    observed samples but also provides access to the desirable properties of
    the maximum likelihood estimator.
    \end{abstract}

    \begin{keyword}[class=MSC]
	\kwd[Primary ]{00X00}
	\kwd{00X00}
	\kwd[; secondary ]{00X00}
    \end{keyword}

    \begin{keyword}
    	\kwd{Dimension reduction}
	    \kwd{importance sampling}
        \kwd{multivariate count data}
        \kwd{Poisson log-normal model}
        \kwd{projected stochastic gradient descent}
    \end{keyword}

\end{frontmatter}

% INTRO ------------------------------------------------------------------------------------------
\section{Introduction}

%% General LVM
Multivariate count data are prevalent in a widening range of applications
such as ecology, genomics, microbiology, astronomy, and economy, just to name a few.
This ubiquity has prompted the development of numerous statistical models, as unlike
continuous multivariate distributions, a generic universal multivariate
distribution for count data does not exist \citep{IYA17}.
Most of the successful proposals are latent variable models belonging to the family of generalized
multivariate mixed models (GMMM).
The latter offers the strength of model-based approaches, enabling the incorporation of external covariates
and allowing the latent variables to be constrained in
various ways to perform a specific task --- regression, variable
selection, and dimension reduction --- while controlling the complexity of the model.
These strengths contribute to the unwavering popularity of these models in the
aforementioned fields of application, where both modeling and interpretability are
essential prerequisites.

%% References on important frameworks including PLNmodels
Significant milestones in the literature include a few generic frameworks initially developed through applications in ecology, where count and abundance tables have long been the norm.
The generalized linear latent variable models (GLLVM) of \cite{Niku2019} are instances of generalized multivariate mixed models with low-dimensional latent variables, where the distribution of observed responses usually belongs to the exponential family.
The Hierarchical Modeling of Species Communities (HMSC), presented in \cite{Ovaskainen2017}, also falls within the class of generalized linear latent variable models with additional layers in the modeling of the latent variables.
Multivariate Poisson Log-Normal models (PLN), as presented in \cite{PLNfrontier}, are yet another instance of the expansive family of generalized linear latent variable models. The latter confers the advantage of developing a generic and versatile framework capable of addressing various tasks, including dimension reduction, regression, clustering, discriminant analysis.

%% Estimation
These modeling frameworks encounter the usual inference issues inherent to
latent variable models.
Specifically, direct and exact likelihood maximization is
difficult since it requires evaluating an integral over a space
of the latent variable dimension.
Hence, direct numerical integration approaches \citep{aitchison1989} are limited to small-scale problems involving solely a few
variables.
Approaches based on Markov chain Monte Carlo (MCMC) techniques
can handle medium-size problems but are computationally expensive
\citep{Hui2016, Tikhonov2020}.
Alternatively, methods based on Laplace approximations exhibit greater computational efficiency but can potentially be inaccurate \citep{INLA}.
An indirect approach relies on the Expectation-Maximization (EM) algorithm, a well-established method for inference in incomplete data models since \cite{dempster1977}.
However, the M-step is practically intractable in GMMM, as it requires computing an expectation with respect to the distribution of the latent variable conditional on the observations.
MCMC techniques can obviously be used in such a setting \citep[\textit{e.g.},][]{karlis2005algorithm}, though displaying the same shortcomings.
More recently, the growing size of datasets and the porting of these methods to other fields of
application where the number of variables expands drastically, such as genomics,
have sparked interest in variational approaches \citep{HOW11, blei2017,
Hui2017,Niku2019} as they provide a good compromise between accuracy and computational
efficiency.

%% PLN variational
In the context of PLN models, which is the focus of this paper,
\cite{PLNfrontier} have extensively used this variational approach in conjunction with the EM algorithm and adapted it to several contexts, including dimension reduction, clustering, sparse covariance.
The implementation provided in \cite{PLNmodels} is efficient and can deal with problems with thousands of observations and hundreds --- even a couple of thousands --- of variables.
The resulting estimator can be shown to converge to the maximum of the surrogate likelihood function and enjoy asymptotical normality \citep{WeM19}.
However, these results pertain to the surrogate model and generally differ from the natural properties of an M-estimator associated with the likelihood \citep{vdV98}. In particular, while the maximum reached by the variational estimator seems at least empirically to coincide with the maximum likelihood estimator (MLE), there is no genuine estimator of the variance of the variational estimator that can be used to measure uncertainty properly:
although the bootstrap method or the jackknife estimator could be used to build an estimator of the variance of the estimate, the variational solution is marred by the lack of relevant statistical guarantees.
Consequently, the design of efficient algorithms that can directly maximize the likelihood and inherit the desirable properties of MLEs is still a key research issue for GMMM, particularly for PLN models.
Such an algorithm allows a more direct assessment of estimator uncertainty by means of asymptotic variance estimates.
It is in this spirit that \cite{stoehr2024} propose a variant of the Monte Carlo EM scheme that combines composite likelihood and importance sampling methods with a focus on applications in synecology.
While the approach benefits from the properties of the maximum composite likelihood estimator, it necessitates splitting the data into overlapping blocks containing a small number of variables.
The mimimum number of blocks required grows quadratically in the number of variables (or species).
As the computational complexity of their Monte Carlo EM increases linearly with the number of blocks,
the solution is primarily suited for problems involving a few dozen variables but does not scale up efficiently to larger problems from a computational perspective.

%% Our proposal
\paragraph*{Contributions}
This paper introduces a projected stochastic gradient descent (SGD) scheme based on self-normalized importance sampling to obtain gradient estimates for optimizing the marginal likelihood of the observed data in the Poisson log-normal model, subject to a rank constraint on the latent space.
This model, introduced by \cite{PLNPCA}, can be seen as a probabilistic version of Principal Component Analysis (PCA) with Poisson emission law, and its standard inference solution is a variational Expectation-Maximization (VEM) algorithm.
Estimating parameters according to the maximum likelihood principle with Monte Carlo simulations is a long-standing
idea for an unnormalized statistical model --- a class of challenging models due to their intractable partition function which is a highly multidimensional integral depending on the parameters.
For instance, Monte Carlo maximum likelihood estimation \citep{geyer1994} uses importance sampling to estimate the partition function while contrastive divergence \citep{hinton2002} estimates the gradient of the log partition function via Monte Carlo methods.
More recently, the noise-contrastive estimation \citep{gutmann2012} reformulates the initial problem to avoid estimating the partition function or its gradient.
Here, we rely on the fact that incomplete data models share similarities with the unnormalized models in that, under mild regularity conditions, the inference resumes to deal with an intractable integral, namely the score function for the observed likelihood \citep{louis82}.
We show that the PLN-PCA model falls within the set of incomplete data models for which the score function is written as an expected value with respect to the conditional distribution of the latent given the observed and can thereby be estimated by simulation methods.
The rank constraint ensures that the importance sampling estimator can handle problems with up to thousands of variables in the emission space, provided that the dimension of the latent space is controlled and limited to some tens.
The projection step onto a convex compact set specifically guarantees that the objective function is $L$-smooth.
We also show that it ensures a bounded mean squared error and bias for the gradient estimator.
Such properties are common in the literature \citep{ghadimi2013stochasticfirstzerothordermethods,mai2021convergence,scaman22a}.

Our major contribution is a novel convergence theorem for the gradient method presented.
To establish the result within the context of a self-normalized importance sampling estimator, we first present a general descent lemma applicable under minimal assumptions --- specifically, $L$-smoothness, and bounded bias and quadratic error for the gradient estimator.
To the best of our knowledge, it is the first result on projected stochastic gradient schemes for potentially both non convex objective functions and random biased gradient estimators.
Given $T$ iterations of our method and $N$ Monte Carlo draws, we obtain a theoretical rate of $O(T^{-\nicefrac{1}{2}} + N^{-1})$ for the gradient mapping norm.
This convergence rate aligns with those reported in the literature,
\BB{including results specific to gradients estimated via Monte Carlo methods
\citep{mohamed2020montecarlogradientestimation,mcbook}}, as well as others
observed in different contexts.
For a non-convex setting with an unbiased gradient estimator, \citet{ghadimi2013stochasticfirstzerothordermethods}
derive an $O(T^{-\nicefrac{1}{2}})$ convergence rate for the gradient's norm objective, while \citet{mai2021convergence} achieve the same rate for the gradient mapping norm but when a projection step is added.
In situations where the gradient estimator is biased, with a bound $b$ on the bias, \citet{biased_SGD} retrieve a rate of $O(T^{-\nicefrac{1}{2}} + b)$, but only when no projection step is performed.

The paper also includes an efficient implementation of our algorithm using \texttt{JAX} library \citep{jax}, and GPU computing.
As a by-product, we provide a \texttt{PyTorch} \citep{pytorch} version of the VEM solution, enabling the analysis of large-scale datasets with hundreds of thousands of observations and tens of thousands of variables.

%% Outline
\paragraph*{Outline}
The paper begins with an introduction to the standard multivariate PLN
model and its PCA version in Section \ref{sec:model}. We then present our stochastic gradient scheme with convergence guarantees in Section \ref{sec:grad-sto}.
In Section \ref{sec:is-proposal}, we propose a sequentially adapted Gaussian mixture distribution to serve
as a valid importance sampling proposal distribution within our algorithm.
Finally, Section~\ref{sec:experiments} details a simulation study on synthetic data and an application to genomic data, where we deal with the problem of dimension reduction and visualization of a transcriptomic single-cell dataset.
Technical details are postponed till Appendices \ref{sec:app-pln}--\ref{sec:app-proposal}.

% NOTATIONS ------------------------------------------------------------------------------------------
\section{Notations and conventions}
%% Euclidian space
Let $p$ and $q$ be positive integers. The vector space of all $p \times q$-matrices over a ring $\mathbb{A}$ is denoted by $\mathcal{M}_{p\times q}(\mathbb{A})$. The subset of all symmetric, positive and definite $p \times p$-matrices over $\mathbb{R}$ is denoted by $\mathcal{S}_{++}^p$.
We denote by $\langle \cdot, \cdot\rangle$ the scalar product on a real
$p$-space $\mathbb{R}^p$, and $\lVert \cdot \rVert$ its associated norm. The
matrix norm induced by $\lVert \cdot \rVert$ on $\mathcal{M}_{p\times
p}(\mathbb{R})$ is also denoted by $\lVert \cdot \rVert$. $\mathrm{Diag}(x)$ is
a diagonal matrix with diagonal equal to $x$ for $x$ a vector. When applied to
matrices or vectors, simple functions like $\log, \exp$ or square apply
element-wise. The vector full of zeros of size $p$ is denoted $\mathbf 0_p$.

%% Measures
We denote by $\mathbb{M}_1(\mathbb{R}^p)$ the set of probability measures on $\mathbb{R}^p$.
% and by $\mathbb{M}_1^+(\mathbb{R}^p)$ the subset of those that are positive almost everywhere.
Given a probability measure $\pi\in\mathbb{M}_1(\mathbb{R}^p)$, $\mathbb{M}_{\pi}$ is the set of probability measures that dominates $\pi$. The product measure $\prod_{i = 1}^n \pi$ on $\mathbb{R}^{d\times n}$ is denoted by $\pi^{\otimes n}$. We use the same notation to refer to a measure and its associated density, meaning that if $\pi$ is absolutely continuous with respect to the Lebesgue measure $\lambda$, $\pi(\mathrm{d} x) = \pi(x) \lambda(\mathrm{d} x)$. The expectation with respect to $\pi$ is denoted by $\mathbb{E}_{\pi}$ . When there is no ambiguity regarding the integration measure, we simply use the notation $\mathbb{E}$.

%% Standard distribution
We denote  by $\mathcal{N}\left(\bmu, \bS \right)$ a $p$-dimensional Gaussian variable with mean $\bmu \in \mathbb{R}^p$ and variance $\mathbf{S} \in \mathcal S_{++}^p$ and $\mathcal{N}(\bx; \bmu, \bS)$ its density evaluated at $\bx \in \mathbb{R}^p$. We denote by $\mathcal{P}(\lambda)$ a Poisson variable with rate $\lambda > 0$.

%% KL
The Kullback--Leibler divergence between $\pi \in \mathbb{M}_1(\mathbb{R}^p)$ and $\mu \in \mathbb{M}_\pi$ is defined by
\begin{equation*}
    \mathrm{KL}(\pi \parallel \mu) = \int \log \frac{\pi(x)}{\mu(x)}  \pi(\mathrm{d} x).
\end{equation*}
%% Entropy
The entropy of a random variable $X$ distributed according to $\pi \in \mathbb{M}_1(\mathbb{R}^p)$ is defined by
\begin{equation*}
\mathcal{H}_{\pi}(X) = -\int \log \pi(x) \pi(\mathrm{d} x).
\end{equation*}

%% L-smooth
Given $\mathcal{X} \subseteq \mathbb{R}^p$, a differentiable function $f:\mathbb R^p\rightarrow \mathbb R$ is said to be $L-$smooth on $\mathcal{X}$ with $L\geq 0$ if its gradient is $L-$Lipschitz on $\mathcal{X}$, namely, for any $\btheta, \btheta' \in \mathcal X$,
\begin{equation*}
\lVert  \nabla_{\btheta} f(\btheta) - \nabla_{\btheta} f (\btheta') \rVert \leq L \lVert \btheta- \btheta' \rVert.
\end{equation*}

% MODEL ------------------------------------------------------------------------------------------
\section{Dimension reduction in multivariate Poisson log-normal models}
\label{sec:model}

%% PLN
\paragraph*{Background: Multivariate Poisson log-normal model}
Consider a data matrix $\bY = (Y_{ij}) \in \mathcal{M}_{n\times p}(\mathbb{N})$ storing $n$ i.i.d. observations of a $p$-dimensional random vector $\bY_i= (Y_{i1}, \dots, Y_{ip})\in\mathbb{N}^p$, $i=1,\dots,n$. The multivariate Poisson lognormal model \citep[see][for its original formulation]{aitchison1989} relates each of the observed vector $\bY_i$ to a latent (or unobserved) $p$-dimensional Gaussian vector $\bZ_i$, whose covariance matrix $\bSigma$ describes the underlying structure of dependence between the $p$ variables.
Following a formalism similar to that of GMMM, the model can also encompass a main effect due to a linear combination of $d$ observed covariates $\mathbf{x}_i\in\mathbb{R}^d$ (including a vector of intercepts), and some possible user-specified offsets $\bo_i = (o_{ij}) \in\mathbb{R}^p$ to take into account the sampling efforts between the samples.
The model assumes that the observations $Y_{ij}$ are independent conditionnally
on $\bZ_i = (Z_{ij})$, and that the conditional distribution $p(Y_{ij} \mid
Z_{ij})$ is a Poisson distribution, namely,
\begin{equation}
\label{eqn:PLN}
    \begin{aligned}
        \{\bZ_i\}_{1 \leq i \leq n}  \text{ \BB{ind.}}:  & & \bZ_i & \sim \mathcal{N}(\BB{\bo_{i} + \bB^\top\bx_i }, \bSigma);
		\\
  	\{Y_{ij}\}_{\substack{1 \leq i \leq n \\ 1 \leq j \leq p}}  \text{ ind.} \mid \{\bZ_i\}_{1 \leq i \leq n}:  & & Y_{ij} \mid Z_{ij} & \sim \mathcal{P}\left(\exp(Z_{ij} )\right),
    \end{aligned}
\end{equation}
where $\bB = [\bB_1, \ldots, \bB_p]\in\mathcal{M}_{d\times p}(\mathbb{R})$ is a latent matrix of regression parameters. In this framework, the main goal is to estimate the vector of parameters $\btheta = (\bB, \bSigma) \in \mathbb{R}^D$, with $D = dp + p(p+1)/2$, from the data matrices $\mathbf{Y}$ and $\mathbf{X}$.
%Estimation of $\theta$ can be done efficiently by variational inference, as explained in \cite{PLNfrontier} and implemented in the package \texttt{PLNmodels} \citep{PLNmodels}.

%% PLN-PCA
\paragraph*{Poisson lognormal-model with low-rank constraint}
Throughout this paper, we focus on the PCA variant of model \eqref{eqn:PLN} as introduced in \cite{PLNPCA}.
The latter is derived by adding a rank constraint on the latent covariance
matrix, such that $\text{rank}(\bSigma) = q < p$. The constraint alleviates the
number of parameters to estimate, which can become prohibitive when the number
of variables $p$ is large in \eqref{eqn:PLN}. This key feature is particularly
relevant in the perspective of importance sampling. Indeed, this allows us to
deal with problems where the dimension of the observation space is potentially
large, in contrast with the number of parameters in the model and the dimension
of the latent space, where the particles are sampled.
The PCA version can be written in a hierarchical framework by adding a layer with, for each individual, a $q$-dimensional standard Gaussian vector $\bW_i$, and introducing an individual-specific linear function $f_{i} : \mathbb{R}^q \rightarrow \mathbb{R}^p$ defined for all $\bw\in\mathbb{R}^q$ as
\begin{equation*}
f_{i}(\bw;\bB, \bC) = \bC\bw + \bB^\top\bx_i + \bo_i,
\end{equation*}
where $\bC\in\mathcal{M}_{p\times q}(\mathbb{R})$ encodes the embedding of the observations into a space of lower dimension. The model is then written as
\begin{equation}
\label{eqn:PLNPCA}
    \begin{aligned}
    	\{\bW_i\}_{1 \leq i \leq n} \textit{ i.i.d.}:  & & \bW_i & \sim \mathcal{N}(\mathbf{0}_q, \mathbf{I}_q);
		\\
        \{\bZ_i\}_{1 \leq i \leq n}  \text{ \BB{ind.}}:  & & \bZ_i & = \bC\bW_i + \bB^\top\bx_i + \bo_i;
		\\
  	\{Y_{ij}\}_{\substack{1 \leq i \leq n \\ 1 \leq j \leq p}}  \text{ ind.} \mid \{\bZ_i\}_{1 \leq i \leq n}:  & & p\left(Y_{ij} \mid Z_{ij}\right) & = \frac{\exp\left\{Y_{ij}Z_{ij}  - \exp(Z_{ij})\right\}}{Y_{ij}!}.
    \end{aligned}
\end{equation}
We refer to Model \eqref{eqn:PLNPCA} as PLN-PCA.
Adopting the PCA terminology, $\mathbf{C}$ is the $p\times q$ matrix of loadings, and $\bW_i$ represents the vector of scores of the $i$-th observation in the low-dimensional latent subspace, whose dimension $q$ corresponds to the number of components.
%An equivalent way of writing this model, where the rank constraint on $\mathbf{\Sigma}$ is explicit, is
%\begin{equation}
%  \label{eqn:PLNrank-model}
%  \mathbf{Y}_i | \mathbf{Z}_i\sim^\text{iid}\mathcal{P}\left(\exp\left\{\mathbf{B}^\top\mathbf{x}_i + \mathbf{Z}_i\right\}\right), \quad \mathbf{Z}_i \sim^\text{iid}  \mathcal{N}({\boldsymbol 0}_p, \mathbf{\Sigma}), \quad \text{with } \mathbf{\Sigma} = \mathbf{C}\mathbf{C}^\top.
%\end{equation}

The vector of unknown parameters to be estimated is now $\btheta = (\bB, \bC) \in \mathbb{R}^d$, with $d=p(q+m)$ that is significantly smaller than $D$ when $q\ll p$, the typical case at play in a context of dimension reduction. The complete log-likelihood of Model \eqref{eqn:PLNPCA} can be written, up to additive constants with respect to model parameters, as
\begin{equation*}
\sum_{i = 1}^n \log p_{\btheta}(\bY_i, \bW_i) = \sum_{i=1}^n \left\lbrace
\langle \bY_i, \bZ_i \rangle - \sum_{j = 1}^p \exp(Z_{ij}) - \frac{1}{2} \lVert \bW_i \rVert^2
%- \frac{q}{2}\log(2\pi)
\right\rbrace.
\end{equation*}
% \begin{align*}
% \sum_{i = 1}^n \log p_{\btheta}(\bY_i, \bW_i)
% & = \sum_{i=1}^n \left\lbrace
% \sum_{j = 1}^p \log {p(Y_{ij} \mid Z_{ij})} +  \log \mathcal{N}(\bW_i\,;\, \mathbf{0}_q, \mathbf{I}_q)
% \right\rbrace
% \\
% & = \sum_{i=1}^n \left\lbrace
% \langle \bY_i, \bZ_i \rangle - \sum_{j = 1}^p \lbrace \exp(Z_{ij}) + \log(Y_{ij}!)\rbrace - \frac{1}{2} \lVert \bW_i \rVert^2 - \frac{q}{2}\log(2\pi)
% \right\rbrace.
% \end{align*}
Without any further assumption on $\bC$, remark that $\btheta$ is not identifiable since the distribution of $\mathbf{Z}_i$ is invariant when multiplying $\mathbf{C}$ by any orthogonal matrix.
However, since $\bY_i$ depends on $\bC$ solely through the covariance $\bSigma = \bC\bC^\top$, it is enough to have the identifiability for the model parametrized by $ (\mathbf{B}, \mathbf{\Sigma})$.
We will discuss this point further when introducing the gradient ascent algorithm.

% ESTIMATION ------------------------------------------------------------------------------------------
\section{Biased stochastic gradient descent}
\label{sec:grad-sto}

\subsection{Solution to maximum likelihood principle}
\label{sec:mle}

Estimation of $\btheta$ is achieved by maximizing the likelihood of the
observed data $p_{\btheta}(\bY)$, or equivalently by solving the optimization
problem
\begin{equation}
\label{eqn:objective}
\btheta^{\mle} = \argmin_{\btheta \in \mathcal{X}} \ell(\btheta),
    \quad
    \ell(\btheta) = -\frac{1}{n} \sum_{i = 1}^n \log p_{\btheta}(\bY_i),
\end{equation}
where $\ell$ is referred to as the loss function. Optimizing such a
function is not straightforward because the marginal
$p_{\btheta}(\bY_i)$ requires integrating out the latent variable
$\bW_i$. We refer to $\btheta^{\mle}$ as the maximum likelihood
estimator.

%% EM algorithm
\paragraph*{Reminder on the EM approach}
The Expectation-Maximization algorithm \citep{dempster1977} circumvent this issue by using a decomposition of the log-likelihood of the observed data $\bY_i$ into
\begin{align}
\label{eqn:EM}
\log p_{\btheta}(\bY_i) & =
\int_{\mathbb{R}^q} \log p_{\btheta}(\bY_i, \bw) p_{\btheta}(\mathrm{d}\bw \mid \bY_i)
-\int_{\mathbb{R}^q}\log p_{\btheta}(\bw \mid \bY_i) p_{\btheta}(\mathrm{d}\bw \mid \bY_i) \nonumber\\
& = \mathbb{E}\left[
\log p_{\btheta}(\bY_i, \bW_i) \mid \bY_i, \btheta
\right] + \mathcal{H}_{ p_{\btheta}(\cdot \mid \bY_i)}(\bW_i),
\end{align}
The algorithm proceeds by evaluating the conditional expectation of the
complete log-likelihood using the current
estimates $\btheta^{(t)}$ of the model parameters:
\begin{equation*}
	Q(\btheta \mid \btheta^{(t)}) =
	\sum_{i = 1}^n \mathbb{E}\left[\log p_{\btheta}(\bY_i, \bW_i) \mid \bY_i, \btheta^{(t)}\right].
\end{equation*}
By iteratively maximizing this quantity, the algorithm generates a sequence
that converges under suitable regularity conditions to the maximum likelihood
estimator $\btheta^\star$ \citep{wu1983, boyles1983}.
However, the conditional $p_\btheta(\bW_i \mid \bY_i)$ is intractable for the
PLN model and its PCA extension.
To address this challenge, variational inference approximates
$p_{\btheta}(\bW_i \mid \bY_i)$ with a surrogate distribution
$\phi_{\boldsymbol{\psi}}(\bV)$ from a parametric family $\mathcal{P}$.
For instance, \cite{CMR18,CMR19} resorted to multivariate Gaussian distributions $\mathcal{N}(\bm^{\vem}_i, \bS^{\vem}_i)$ with diagonal covariance matrix.
The Variational EM method alternates between updates of
the variational parameter $\boldsymbol{\psi}$ and the model parameter
$\btheta$, aiming to maximize a lower bound of the log-likelihood, defined as
\begin{equation*}
\log p_\btheta(\bY_i) - \mathrm{KL}\left[\phi_{\boldsymbol{\psi}}(\bW_i) \Vert p_\btheta(\bW_i \mid \bY_i)\right]
= \mathbb{E}_{\phi_{\boldsymbol{\psi}}}\left[\log p_{\btheta}(\bY_i, \bV)\right] + \mathcal{H}_{\phi_{\boldsymbol{\psi}}}(\bV).
\end{equation*}
At convergence, the variational solution is associated with
\begin{equation}
\label{eqn:variational-sol}
\phi_{\boldsymbol{\psi}} = \argmin_{\phi \in \mathcal{P}} \mathrm{KL}\left[\phi \Vert  p_{\btheta}(\cdot \mid \bY_i)\right].
\end{equation}

%% Stochastic gradient
\paragraph*{Stochastic gradient scheme}
In contrast to existing methods, we propose to address the optimization problem \eqref{eqn:objective} directly with an SGD scheme. Our approach leverages that the Louis principle applies to Model \eqref{eqn:PLNPCA}, as outlined below (see Appendix \ref{sec:app-pln} for a proof).

%------------------------------------------------------------------------------------------
%------------------------------------------------------------------------------------------
\begin{restatable}{prop}{propscore}
%\begin{prop}
\label{prop:score}
For all individual $i = 1, \ldots, n$, the incomplete log-likelihood $\btheta \mapsto \log p_{\btheta}(\bY_i)$ of Model \eqref{eqn:PLNPCA}
is twice continuously differentiable on $\mathbb{R}^d$ and its score function can be written as
\begin{equation}
\label{eqn:score-marginal}
s_i(\btheta)
= \int_{\mathbb{R}^q} \nabla_{\btheta}\log p_{\btheta}(\bY_i, \bw)
p_{\btheta}(\mathrm{d}\bw \mid \bY_i)
= \mathbb{E}\left[
\nabla_{\btheta}\log p_{\btheta}(\bY_i, \bW_i) \mid \bY_i, \btheta
\right].
\end{equation}
%\end{prop}
\end{restatable}
%------------------------------------------------------------------------------------------
%------------------------------------------------------------------------------------------

While the intractability of the conditional distribution renders exact calculation infeasible even for a small value of $q$, identity \eqref{eqn:score-marginal} is instrumental in estimating the score function with Monte Carlo methods and designing an SGD scheme. Given a learning rate $\gamma\in\mathbb{R}_{+}^*$ and an initial point $\btheta^{(1)}\in\mathbb{R}^d$, the SGD scheme
recursively defines a sequence $\{\btheta^{(t)}\}_{t\in\mathbb{N}^*}$ through
the equation
\begin{equation}
\label{eqn:sgd}
\btheta^{(t+1)} = \btheta^{(t)} - \gamma \widehat{g}^{(t)}, %t\in\mathbb{N}, \tag{SGD}
\end{equation}
where $ \widehat{g}^{(t)}$ is a possibly biased estimator of $\nabla_{\btheta}\ell(\btheta^{(t)})$.
Here we explore the opportunity of importance sampling methods \citep{kahn1949, kahn1951} to define $\widehat{g}^{(t)}$.
Indeed, the lack of closed-form for $p_{\btheta}(\bW_i\mid\bY_i)$ hinders Monte Carlo methods that rely on exact samples from it.
However, importance sampling overcomes this difficulty by changing the integration measure.
For any $\btheta\in\mathbb{R}^d$, it approximates $p_{\btheta}(\cdot \mid
\bY_i)$ with a random probability measure based on weighted samples from a
probability density function $\nu_i(\cdot\,;\btheta)$, possibly depending on
$\btheta$ and referred to as proposal distribution, such that
$p_{\btheta}(\cdot \mid \bY_i)$ is absolutely continuous with respect to
$\nu_i(\cdot\,;\btheta)$.
The importance sampling method is then based on the following identity
\begin{equation*}
s_i(\btheta) = \int_{\mathbb{R}^q} \frac{p_{\btheta}(\bv \mid \bY_i)}{\nu_i(\bv;\btheta)}\nabla_{\btheta}\log p_{\btheta}(\bY_i, \bv)
\nu_i(\mathrm{d}\bv;\btheta).
\end{equation*}
To circumvent the issue of evaluating the intractable distribution in the above, the method leverages a tractable non-normalized version of the conditional distribution, namely the joint distribution. Let us introduce the Radon-Nikodym derivative of $p_{\btheta}(\bY_i, \cdot)$ with respect to $\nu_i(\cdot\,;\btheta)$:
\begin{equation}
\label{eqn:rho}
\rho_{\btheta,i}(\bv) = \frac{p_{\btheta}(\bY_i, \bv)}{\nu_i(\bv;\btheta)},
\quad\bv\in\mathbb{R}^q.
\end{equation}
The score can be written as
\begin{equation*}
s_i(\btheta) =
\int_{\mathbb{R}^q} \rho_{\btheta, i}(\bv) \nabla_{\btheta}\log p_{\btheta}(\bY_i, \bv)
\nu_i(\mathrm{d}\bv;\btheta)
\left/
\int_{\mathbb{R}^q} \rho_{\btheta, i}(\bv) \nu_i(\mathrm{d}\bv;\btheta)
\right..
\end{equation*}
Thereby, given $N\in\mathbb{N}^*$ independent samples $\bv_{i,1}, \ldots,
\bv_{i, N}$ from $\nu_i(\cdot\,;\btheta)$, the self-normalized importance
sampling (SNIS) estimator of $s_i(\btheta)$ is
\begin{equation}
\label{eqn:is}
\widehat{s}_i^N(\btheta) = \sum_{r = 1}^N \rho_{\btheta,i}(\bv_{i, r}) \nabla_{\btheta} \log p_{\btheta}\left(\bY_i, \bv_{i, r}\right) \left/ \sum_{s = 1}^N \rho_{\btheta, i}(\bv_{i,s}) \right..
\end{equation}
A possible solution is then to define $\widehat{g}^{(t)}$ by averaging the estimator \eqref{eqn:is} across individuals.
In what follows, we rely on a mini-batch strategy where we use a single individual at a time.
Specifically, at iteration $t\in\mathbb{N}^*$, we draw an individual $i(t)$ uniformly in $\{1, \ldots, n\}$ and
the gradient estimator within the update Equation \eqref{eqn:sgd} is
\begin{equation}
\label{eqn:grad-hat}
\widehat{g}^{(t)} = -\widehat{s}^N_{i(t)}\cp(\btheta^{(t)}\cp).
\end{equation}

% CONVERGENCE ------------------------------------------------------------------------------------------
\subsection{Convergence guarantees}

\begin{algorithm2e}[t!]
\caption{Importance Sampling based Gradient Descent (ISGD)}
\label{algo:sgis}
\KwIn{initial point $\theta^{(1)}$,  learning rate $\gamma$, number of iterations $T$, number of Monte Carlo draws $N$.}
\KwOut{the sequence $\btheta^{(1)}, \dots, \btheta^{(T+1)}$}
\For{$t = 1$ \KwTo $T$}{
    Sample $i$ uniformly in $\{ 1, \dots, n\}$\;
    Sample $(\bv_{i, 1}, \ldots, \bv_{i, N})$ from $\nu_{i}^{\otimes N}\cp(\cdot\,;\btheta^{(t)}\cp)$\;
    Update $\btheta^{(t+1)} = P_{\Theta}\cp(\btheta^{(t)} - \gamma \widehat{g}^{(t)}\cp)$ with $\widehat{g}^{(t)} = -\widehat{s}^N_{i}\cp(\btheta^{(t)}\cp)$ as in Equation \eqref{eqn:is}\;
}
\end{algorithm2e}

The SNIS estimator \eqref{eqn:is} is strongly consistent, but exhibits bias for a fixed sample size $N$.
As shown in this section, controlling this bias is essential to ensure the convergence of our algorithm.
This requires imposing further constraints on the optimization problem.
In the remainder of the paper, let $\Theta \subset \mathbb{R}^d$ be a nonempty, compact, and convex set.
In place of the standard update \eqref{eqn:sgd}, we employ a projected SGD algorithm, which is defined by
\begin{equation}
\label{eqn:proj-sgd}
\btheta^{(t+1)} = P_{\Theta}\cp(\btheta^{(t)} - \gamma \widehat{g}^{(t)}\cp),
\end{equation}
where $P_{\Theta}$ stands for the orthogonal projection on $\Theta$. Algorithm \ref{algo:sgis} summarizes the overall scheme.
Formally, the projection step ensures that all iterates remain bounded, which in turn guarantees the $L$-smoothness and the bounded gradient conditions commonly assumed in the literature to establish convergence properties.

\paragraph*{Preliminary result for $L$-smooth functions}
In the following, we delve into the analysis of an SGD algorithm (i) for non-convex and constrained optimization with $L$-smooth objective, (ii) with a gradient estimator that is both biased and random. To the best of our knowledge, no established convergence theorem exists for these conditions. \citet[][Lemma 3.2]{mai2021convergence} achieve a similar result but for unbiased gradient estimator.
Let $F:\mathbb R^{d} \rightarrow \mathbb R$ be defined by
\begin{equation}
\label{eqn:F}
F(\btheta) = \ell(\btheta) + I_{\Theta}(\btheta),
\quad
I_{\Theta}(\btheta) =
\begin{cases}
    0 & \text{if } \btheta \in \Theta, \\
    \infty & \text{if } \btheta \notin \Theta.
\end{cases}
\end{equation}
For a real $\eta > 0$, its Moreau envelope is
\begin{equation*}
F_{\eta}(\btheta) = \inf_{\btheta' \in \mathbb{R}^d} \left\{F(\btheta') + \frac 1{2 \eta} \lVert \btheta-\btheta' \rVert ^2 \right\}.
\end{equation*}
The Moreau envelope is useful from an optimization perspective, as it has the same set of minimizers as $F$, while being differentiable, unlike $F$.
This characteristic is instrumental in analyzing the convergence of proximal gradient methods such as the one from Equation \eqref{eqn:proj-sgd}.
Our following result yields control over the Moreau envelope, providing we have a $L$-smooth
loss function and a gradient estimator $\widehat{g}^{(t)}$ with bounded mean squared error and bias (see Appendix \ref{sec:app-optim} for proof).

\begin{restatable}[Descent lemma]{lemma}{lemcontrol}
\label{lem:control}
Let consider the gradient scheme as defined by Equation \eqref{eqn:proj-sgd}.
Assume that
\begin{enumerate}
\item the function $\ell$ is $L$-smooth on $\Theta$, and denote $\Gamma = \sup_{\btheta \in \Theta} \lVert \nabla_{\btheta} \ell(\btheta) \rVert$;
%\label{lem:control:i)
\item for $t \in \mathbb{N}^*$, one has
\begin{align*}
\sigma^{(t)} & = \mathbb{E} \left[
\cp\lVert
\widehat{g}^{(t)} - \nabla_{\btheta} \ell\cp( \btheta^{(t)} \cp)
\cp\rVert^2
\mid \btheta^{(t)}
 \right] < \infty,
\\
\xi^{(t)} & = \left\lVert
\mathbb{E} \left[
\widehat{g}^{(t)}\mid \btheta^{(t)}
 \right]
- \nabla_{\btheta} \ell\cp( \btheta^{(t)} \cp)
 \right\rVert < \infty.
\end{align*}
%\label{lem:control:ii)
\end{enumerate}
Then, for any real constant $\eta \in \big(0, 1/\max\{2\Gamma + L, 2L\}\big]$,
\begin{equation}
\label{eq:lyapunov_biased}
\begin{split}
\mathbb{E} \left[
F_\eta\cp( \btheta^{(t+1)} \cp) \mid \btheta^{(t)}
 \right]
\leq F_{\eta}\cp( \btheta^{(t)} \cp)
&
- \frac{\gamma}{2} \cp\lVert \nabla F_{\eta}\cp( \btheta^{(t)} \cp) \cp\rVert^2
\\
&
+ \frac{\gamma + \gamma^2\Gamma}{\eta} \xi^{(t)}
+ \frac{\gamma^2}{2\eta} \cp( \sigma^{(t)} + \Gamma^2 \cp).
\end{split}
\end{equation}
\end{restatable}
Remark that $\Gamma$ is finite as the gradient is $L$-Lipschitz continuous on a bounded set. Consequently the interval $\big(0, 1/\max\{2\Gamma + L, 2L\}\big]$ is nonempty.

\paragraph*{Applicability to the PLN-PCA model}
We now demonstrate that the loss function for the PLN-PCA model and the importance sampling estimate \eqref{eqn:grad-hat} satisfy the assumptions from Lemma \ref{lem:control}.
\begin{restatable}{prop}{propsmoothness}
\label{prop:smoothness}
Under Model \eqref{eqn:PLNPCA}, for any nonempty compact subset $\Theta\subset\mathbb{R}^d$, we have that
\begin{enumerate}
\item for any individual $i = 1, \ldots, n$, there exists a real
\begin{equation*}
0 < \zeta_i = \inf_{\theta\in\Theta} p_{\btheta}(\bY_i);
\end{equation*}
\item there exists a real $L \geq 0$ such that the objective function $\ell$, as defined in \eqref{eqn:objective}, is $L$-smooth on $\Theta$.
\end{enumerate}
\end{restatable}

Regarding Assumption $(ii)$ of Lemma \ref{lem:control}, we should note that both $\sigma^{(t)}$ and $\xi^{(t)}$ are driven by, respectively, the mean squared error and the bias of the importance sampling estimate \eqref{eqn:is}.
\citet[][Theorem 2.3]{agapiou2017} yield sufficient conditions to have control over the bias and error.
The latter led to the subsequent assumptions on the proposal distribution for our optimization problem.
\begin{hyp}
\label{hyp:proposal}
For all individual $i \in \{1, \ldots, n\}$, the proposal distribution $\nu_i$ is chosen such that
\begin{align}
\tag{A1.1}\label{eqn:A11}
\lambda_i & = \sup_{(\btheta, \bV)\in\Theta\times\mathbb{R}^q}
\rho_{\btheta,i}(\bV)< \infty,
\\
\tag{A1.2}\label{eqn:A12}
\beta_i & = \sup_{\btheta\in\Theta} \mathbb{E}_{\nu_i(\cdot\,;\btheta)}\left[
\left\lVert
\nabla_{\btheta} \log p_{\btheta}(\bY_i, \bV)
\right\rVert_1^4
\right]  < \infty.
\end{align}
\end{hyp}
Following the work of \cite{agapiou2017}, we obtain finite bound when integrating the mean squared error and the bias of the importance sampling estimate \eqref{eqn:is} with respect to the random mini-batch index $i(t)$ (see Appendix \ref{sec:app-convergence} for proof).
\begin{restatable}{prop}{propagapiou}
\label{prop:agapiou}
% \textcolor{orange}{
    Let $\{i(t)\}_{t\in\mathbb{N}^*}$ and $\{\btheta^{(t)}\}_{t\in\mathbb{N}^*}$ be the associated
random sequences generated by Algorithm \ref{algo:sgis}.
Under Model \eqref{eqn:PLNPCA}, if Assumption \ref{hyp:proposal} holds, then,
for all $t \in \mathbb N$,
% }
% Let $\{i(t)\}_{t\in\mathbb{N}^*}$ be a sequence of independent uniform draws in
% $\{1, \ldots, n\}$ and $\{\btheta^{(t)}\}_{t\in\mathbb{N}^*}$ the associated
% random sequence generated by Algorithm \ref{algo:sgis}.
% Under model \eqref{eqn:PLNPCA}, if Assumption \ref{hyp:proposal} holds, then,
% for any $t\in\mathbb{N}$
\begin{align*}
\overline{\sigma}_{\rm IS}^{(t)} & = \mathbb{E}\left[
\left\lVert
\widehat{s}^{N}_{i(t)}\cp( \btheta^{(t)} \cp) - \nabla \log p_{\btheta^{(t)}}\left(\bY_{i(t)}\right)
\right\rVert^2
\mid \btheta^{(t)}
\right] \leq
\frac{d}{N}M_{\sigma},
%\frac{d}{N}M_\sigma,
\\
\overline{\xi}_{\rm IS}^{(t)} & =
\mathbb{E} \left[
\left\lVert
\mathbb{E}\left[
\widehat{s}^{N}_{i(t)}\cp( \btheta^{(t)} \cp)  -  \nabla \log p_{\btheta^{(t)}}\left(\bY_{i(t)}\right)
\mid \btheta^{(t)}, i(t)
\right]
\right\rVert
\mid \btheta^{(t)}
\right]
 \leq \frac{\sqrt{d}}{N}M_\xi,
\end{align*}
where $M_\sigma$ and $M_\xi$ are two finite and positive constants given by
\begin{align*}
M_{\sigma} & =
\frac{12}{n}\sum_{i = 1}^n \frac{\lambda_i^2\sqrt{\beta}_i}{\zeta_i^2}
\left(
1 + \frac{250\lambda_i}{\zeta_i} + \frac{9\lambda_i^2}{\zeta_i^2}
\right),
\\
M_{\xi} & = \frac{4}{n}
\sum_{i = 1}^n
\frac{\lambda_i^2\beta_i^{1/4}}{\zeta_i^2}
\left\lbrace
2
+ \sqrt{3
\left(
1 + \frac{250\lambda_i}{\zeta_i} + \frac{9\lambda_i^2}{\zeta_i^2}
\right)
}
\right\rbrace,
\end{align*}
with $\lambda_i$ and $\beta_i$ as in Assumption \ref{hyp:proposal}.
\end{restatable}
As detailled in the proof of Theorem \ref{thm:convergence} in Appendix \ref{sec:app-convergence}, this result implies the conditions required on $\sigma^{(t)}$ and $\xi^{(t)}$, since there is a constant $A$ such that
\begin{equation*}
\sigma^{(t)}  \leq A\left(\overline{\sigma}_{\rm IS}^{(t)} + \overline{\xi}_{\rm IS}^{(t)} + 1\right),
\quad\text{and}\quad
\xi^{(t)}  \leq \overline{\xi}_{\rm IS}^{(t)}.
\end{equation*}

\paragraph*{Convergence of the gradient mapping}
As the local minimum of the loss function may lie outside the compact set $\Theta$, we cannot prove that the norm of $\nabla \ell(\theta)$ becomes arbitrarily small within Algorithm \ref{algo:sgis}.
Instead, in the context of gradient methods incorporating a projection step, the convergence rate is characterized in terms of the norm of the gradient mapping $G_{\eta}^{(t)}$ \citep{smoothmap} defined for any real $\eta > 0$ by
\begin{equation}
\label{eq:gradient-mapping}
G_{\eta}^{(t)} = \frac{\btheta^{(t)} - P_{\Theta}\cp(\btheta^{(t)} - \eta \nabla \ell\cp(\btheta^{(t)}\cp)\cp)}{\eta}.
\end{equation}
This mapping is a tailored gradient objective, specifically modified to handle the projection step, whose norm is equivalent to the gradient norm of the Moreau envelope (see Lemma \ref{lem:Moreau_mapping}).
Our next result shows that the norm of the gradient mapping for Algorithm \ref{algo:sgis} can be rendered arbitrarily small, provided we use a sufficiently large number of iterations $T$ and particles $N$.

\begin{restatable}{thm}{thmconvergence}
\label{thm:convergence}
Let $\btheta^{(1)}\in\Theta$ be an initial value and $\gamma_0\in\mathbb{R}_{+}^*$ a user-specified initial learning rate.
Under Model \eqref{eqn:PLNPCA}, if Assumption \ref{hyp:proposal} holds, then
for any $T\in\mathbb{N}^*$ and any real constant $\eta \in(0, 1/\max\{2\Gamma,
L\}]$, the sequence $\{\btheta^{(t)}\}_{1\leq t\leq T}$ defined by
Algorithm \ref{algo:sgis} with $\gamma = \gamma_0/\sqrt{T}$ satisfies
\begin{dmath}\label{eq:thm_conv}
\frac{1}{T}\sum_{t = 1}^T \mathbb{E}\left[ \cp\lVert G_{\eta}^{(t)} \cp\rVert^2 \right]
\leq
\frac{2\tau}{\gamma_0(L\eta + 1) \sqrt{T}}
\left(
\ell\cp( \btheta^{(1)} \cp)
- \ell(\btheta^{\mle})
+ \frac{\gamma_0^2(L\eta + 1)}{2\eta}
\left[
\Delta^2 + \Gamma^2 + \frac{d}{N} \left\lbrace M_{\sigma}  + \frac{2(\Delta + \Gamma)}{\sqrt{d}} M_{\xi} \right\rbrace
\right]
\right)
+ \frac{2\tau\sqrt{d}}{\eta N} M_{\xi},
\end{dmath}
with $L$ the smoothness constant of $\ell(\cdot)$, $\Gamma = \sup_{\btheta \in \Theta} \rVert \nabla \ell(\btheta)\rVert$, constants $M_\sigma$ and $M_{\xi}$ as defined in Proposition \ref{prop:agapiou}, and
\begin{equation*}
\begin{split}
\tau = \frac{(2L\eta + 1)^2}{L\eta+1} \left( 1 + \sqrt{\frac{L\eta}{L\eta + 1}}  \right)^2,
\Delta = \max_{i = 1, \ldots, n} \sup_{\btheta\in\Theta} \left\lVert
\nabla_{\btheta}\ell\left(\btheta\right) + \nabla_{\btheta}\log p_{\btheta}\left( \bY_{i}\right)
\right\rVert.
\end{split}
\end{equation*}
\end{restatable}

\begin{remark}
While the upper bound $1/\max\{2\Gamma,L\}$ is unknown, Algorithm \ref{algo:sgis} does not depend on the choice of $\eta$ in practice, and therefore neither $\Gamma$ nor $L$ need to be estimated.
\end{remark}

\begin{remark}
    \BB{We acheived a convergence rate of $O(T^{\nicefrac{-1}{2}} + N^{-1})$. The convergence rate so obtained does not account for the sampling cost of Monte Carlo draws. If ones want to take into account the computational cost, the convergence rate becomes $O(N^{\nicefrac{1}{2}}T^{-\nicefrac{1}{2}} + N^{-1})$.}
\end{remark}

% IS_PROPOSAL------------------------------------------------------------------------------------------
\section{Importance sampling proposal choice}
\label{sec:is-proposal}

In this section, we propose a specific choice of proposal distribution that satisfies Assumption \ref{hyp:proposal}, namely a mixture distribution.
Mixture distributions are often chosen for this task \citep[\textit{e.g.},][]{cappe2008} because of their flexibility as parametric models.
In what follows, we leverage the model structure and focus on two-component Gaussian mixture distributions.
Given two real constants $\alpha\in[0, 1]$ and $\delta > 0$, denote the two-component Gaussian mixture with mean $\bmu\in\mathbb{R}^q$ and covariance $\bS\in\mathcal{S}_{++}^q$ by
\begin{equation}
\label{eqn:mixt-dist}
\mathcal{GM}(\cdot;\bmu, \bS, \alpha, \delta)
= (1 - \alpha) \mathcal{N}(\cdot;\bmu, \bS) + \alpha \mathcal{N}(\cdot; \bmu, \delta \mathbf{I}_q).
\end{equation}
For each individual $i=1, \ldots, n$, the proposal distribution $\nu_i(\cdot\,;\btheta)$ is set to such a mixture, and $\bmu$ and $\bS$ are iteratively adapted according to the current estimate $\btheta^{(t)}$ (see Algorithm \ref{algo:sgis-adaptive}).
However, to ensure convergence, it is necessary to impose conditions on \eqref{eqn:mixt-dist} (see Appendix \ref{sec:app-proposal} for proof).
%---
\begin{restatable}{lemma}{lemmamixt}
\label{lemma:valid_mixt}
Let $\alpha\in(0, 1]$ and $\delta > 1$. If for any $i = 1, \ldots, n$,
$\btheta \mapsto \bmu_i(\btheta)\in\mathbb{R}^q$ and $\btheta \mapsto
\bS_i(\btheta)\in\mathcal{S}_{++}^q$ are continuous on $\Theta$, then the
proposal distribution defined by
\begin{equation}
\label{eqn:prop-mixt}
\nu_i(\cdot\,;\btheta) = \mathcal{GM}(\cdot;\bmu_i(\btheta), \bS_i(\btheta), \alpha, \delta)
\end{equation}
fulfils Assumption \ref{hyp:proposal}.
\end{restatable}
%---
Constraining the parameter $\delta$ to live in $(1, +\infty)$ is sufficient to guarantee that the Radon--Nikodym derivative $\rho_{\btheta, i}$ is uniformly bounded with respect to $\btheta \in \Theta$.
However, this does not ensure an efficient gradient estimator in terms of error or bias.
The bias and the error are both related to the Kullback--Leibler divergence $\mathrm{KL}[p_{\btheta}(\cdot\mid\bY_i) \Vert \nu_i(\cdot\,;\btheta)]$ \citep{agapiou2017, CD18}.
Specifically, these results show that the estimator exhibits enhanced efficiency
for a fixed computational budget, as the Kullback--Leibler divergence decreases.
Adaptive importance sampling addresses the minimization of $\nu \mapsto
\mathrm{KL}[p_{\btheta}(\cdot\mid\bY_i) \Vert \nu]$ over a given class of
probability measure.
For instance, the Population Monte Carlo proposed by \cite{cappe2008} provides a solution when the proposal is a mixture distribution.
Nevertheless, implementing such methods within an SGD scheme can
be computationally intensive. Indeed, the target distribution of the adaptive
scheme changes at each iteration of the gradient scheme, necessitating a full
run of the adaptive method at each iteration.

\paragraph*{Practical implementation}
In the following, we present a simpler heuristic that is efficient for the class of problems presented in this paper, albeit not optimal in terms of the Kullback--Leibler divergence.
We consider using the mean and covariance of the conditional distribution $p_{\btheta}(\bW_i \mid \bY_i)$, namely
\begin{equation}\label{eq:moments}
\bmu_i(\btheta) = \mathbb{E} [\bW_i \mid \bY_i, \btheta],
\quad
\bS_i(\btheta) = \mathbb{V} [\bW_i \mid \bY_i, \btheta].
\end{equation}
Both functions are continuous on $\Theta$ (see Appendix \ref{sec:app-proposal}).
Moreover, the parameter $\alpha$ can be interpreted as a regularization parameter.
Indeed, in the limiting case $\alpha = 0$, the
proposal distribution as defined in
\eqref{eqn:prop-mixt} resumes to the optimal
Gaussian proposal distribution, that is
\begin{equation}
    \label{eq:proposal-opt}
\nu_i(\cdot\,;\btheta) = \argmin_{\nu\in\mathcal{F}} \mathrm{KL}[p_{\btheta}(\cdot\mid\bY_i) \Vert \nu],
\quad
\mathcal{F} = \left\{
\mathcal{N}(\bmu, \bS) \,;\, \bmu\in\mathbb{R}^q,\, \bS \in\mathcal{S}_{++}^{q}
\right\}.
\end{equation}
While the Radon--Nikodym derivative $\rho_{\btheta, i}$ with respect to such a proposal may not necessarily be bounded for any $\btheta \in \Theta$, it points out the effect of $\alpha$.
The mixture distribution \eqref{eqn:prop-mixt} balances a component that
informs on the intractable conditional distribution and a regularization or
defensive component that plays a similar role to that of \cite{cappe2008}.
Practically speaking, we should opt for a small value of $\alpha$ to improve
the efficiency of the importance sampling.

Both $\bmu_i(\btheta)$ and $\bS_i(\btheta)$ are unknown and must be estimated.
Obviously, we can use the importance sampling method, since we could simply recycle the particles simulated to estimate the gradient.
However, such a solution may lead to poorly conditioned and non-positive definite matrix estimates for $\bS_i(\btheta)$.
A more robust alternative can be achieved using the Hessian of the log-complete likelihood:
\begin{align}
    \bS_i^{H}(\btheta)  & = - \left[\nabla^2 _{\bw} \log p_{\btheta}\left(\bY_i, \bw\right)_{|\bw= \mathbf \bmu_i(\btheta)}\right]^{-1} \\
                        & = \left[\mathbf{I}_q +\mathbf{C}^{\top }\operatorname{Diag}[ \exp \lbrace f_{i}(\bmu_i(\btheta); \bB, \bC) \rbrace]\mathbf{C}  \right]^{-1}. \label{eqn:curv}
\end{align}
This alternative stems from the second order Taylor expansion of the complete
log-likelihood, and has been used in various contexts, such as posterior
approximation \citep{tierney1986accurate} and importance sampling \citep[Chapter 9]{mcbook}.
Interestingly, for a Gaussian distribution, the Taylor expansion exactly
relates the curvature of the scalar field at its mode to the variance, namely
$\bS_i^{H}(\btheta)$ corresponds to the variance.
In contrast to a Monte Carlo estimate of the covariance of the conditional distribution, it directly follows from Equation \eqref{eqn:curv} that the Hessian of the log-complete likelihood, and consequently its inverse, is definite positive. Moreover, Lemma \ref{lemma:valid_mixt} also applies to this choice of covariance matrix, since the function $\btheta
\mapsto \bS_i^{H}(\btheta)$ is continuous as a composition of continuous functions.

% The robust variance approximation has been widely used in Laplace
% approximations \citep{LA1993Breslow,LA2004Huber}.

%---
\begin{algorithm2e}[t!]
\caption{Adaptive Importance Sampling based Gradient Descent %(AISGD)
}
\label{algo:sgis-adaptive}
\KwIn{initial point $\theta^{(1)}$,  learning rate $\gamma$, number of iterations $T$, number of Monte Carlo draws $N$, mixture parameter $0<\alpha\leq 1$, parameter $\delta>0$.}
\KwOut{the sequence $\btheta^{(1)}, \dots, \btheta^{(T+1)}$}
\For{$t = 1$ \KwTo $T$}{
    Sample $i$ uniformly in $\{ 1, \dots, n\}$\;
    Compute the estimate $\widehat{\bmu}_i$ of $\bmu_i(\btheta^{(t)})$ as defined in Equation \eqref{eq:moments} %via SNIS
    \;
    Compute the estimate $\widehat{\bS}_i$ of $\bS_i(\btheta^{(t)})$ as defined in Equations \eqref{eq:moments} or \eqref{eqn:curv}
%    via
%    $\begin{cases} \text{SNIS} &\text{(AISGD)} \\ \text{Robust covariance as in
%        Equation \ref{eq:curvature}}&\text{(Robust-AISGD)}
%    \end{cases}$
    \;
    Set $\nu_i = \mathcal{GM}(\cdot;\widehat{\bmu}_i, \widehat{\bS}_i, \alpha, \delta)$\;
    Sample $(\bv_{i, 1}, \ldots, \bv_{i, N})$ from $\nu_{i}^{\otimes N}\cp(\cdot\,;\btheta^{(t)}\cp)$\;
    Compute $\btheta^{(t+1)} = P_{\Theta}\cp(\btheta^{(t)} - \gamma \widehat{g}^{(t)}\cp)$ with $\widehat{g}^{(t)} = -\widehat{s}^N_{i}\cp(\btheta^{(t)}\cp)$ as in Equation \eqref{eqn:is}\;
}
\end{algorithm2e}

% NUMERIC------------------------------------------------------------------------------------------
\section{Simulation study}
\label{sec:experiments}

\paragraph{Competitors} We compare different variants of our algorithm corresponding to specific choices of the proposal distribution $\nu_i(\cdot\,;\btheta)$:
\begin{itemize}%[leftmargin = *]
\item \textbf{ISGD-VEM:} we set the proposal distribution to the variational distribution $\phi_\psi$, as defined in Equation \eqref{eqn:variational-sol}.
Although it differs from the optimal proposal distribution
\eqref{eq:proposal-opt}, it represents a natural choice as it serves as the
optimal Gaussian surrogate for the conditional distribution in terms of the
Kullback--Leibler divergence $\phi \mapsto \mathrm{KL}[\phi\Vert
p_{\btheta}(\cdot\mid \bY_i)]$.
% ---------
\item \textbf{ISGD-VEMmix:} there is no guarantee that the Radon--Nikodym
    derivative with respect to the VEM proposal is bounded, and thus ensures
    the convergence of Algorithm \ref{algo:sgis}. To address this, we introduce
    a defensive component in this version and consider
\begin{equation*}
\nu_i(\cdot\,;\btheta) = \mathcal{GM}(\cdot\,; \bm^{\vem}_i, \bS^{\vem}_i, \alpha, \delta),
\end{equation*}
where $\bm^{\vem}_i$ and $\bS^{\vem}_i$ are the mean and covariance of the variational distribution $\phi_\psi$.
% ---------
\item \textbf{AISGD-SNIS:} it corresponds to Algorithm \ref{algo:sgis-adaptive} and the choice of the mixture distribution
\begin{equation*}
\nu_i(\cdot\,;\btheta) = \mathcal{GM}(\cdot\,; \widehat{\bm}_i(\btheta), \widehat{\bS}_i(\btheta), \alpha, \delta),
\end{equation*}
where $\widehat{\bm}_i(\btheta)$ and $\widehat{\bS}_i(\btheta)$ are SNIS estimators of the mean and the covariance of the conditional distribution $p_{\btheta}(\cdot\mid\bY_i)$.
% ---------
\item \textbf{AISGD-Hessian:} it corresponds to Algorithm \ref{algo:sgis-adaptive} and the choice of the mixture distribution
\begin{equation*}
\nu_i(\cdot\,;\btheta) = \mathcal{GM}(\cdot\,; \widehat{\bm}_i(\btheta), \bS^{H}_i(\btheta), \alpha, \delta),
\end{equation*}
with $\bS_i^{H}(\btheta)$, as defined in Equation \eqref{eqn:curv}.
\end{itemize}
All these methods are initialized with the variational estimator $\btheta^{\vem}$
fitted with the standard VEM algorithm implemented in
\texttt{pyPLNmodels}\footnote{\url{https://github.com/PLN-team/pyPLNmodels}}.
Throughout the numerical study, the mixture hyperparameters are set to $\alpha = 0.001$ and $\delta = 1.1$.

\BB{
    Note that the convergence result stated in Theorem \ref{thm:convergence} apply to all the aforementionned competitors but the ISGD-VEM method. Indeed, the proposal used by ISGD-VEM generally does not satisfy Assumption \ref{eqn:A11} whereas the latter assumption is fullfilled for the other algorithms due to the the continuity of the mean and variance of the
proposals with respect to $\btheta$, combined with Lemma \ref{eqn:prop-mixt}. Therefore these methods theoretically converge to a critical point  up to a bias that depends on the choice of the proposal
distribution and the number of Monte Carlo draws $N$. Similarly to EM procedures, the methods converge to the MLE provided that the initialization point is appropriately chosen.
Although we could initialize the method from various starting points, we found that using variational estimates as initialization, combined with the inherent stochasticity of the method, consistently led to this desirable outcome.
}

\subsection{Synthetic data}
\label{sec:synthetic-data}

\paragraph{Data generation} We consider simulation settings with $n = 300$
individuals, $p = 150$ variables, $d = 1$ covariate (that is, one intercept),
and rank constraints $q = 3, 5$ and $15$.
The offset term $\bo$ is set to zero.
For each value of $q$, we sample $M = 100$ datasets $\bY^{(q,m)}$, $m = 1, \ldots, M$, according to the PLN-PCA model \eqref{eqn:PLNPCA} with the following regression parameters
\begin{equation*}
B^{\star}_{kj} \sim \mathcal{N}\left(2, 1\right),\quad k =1,\dots,m,\, j = 1, \dots, p.
\end{equation*}
The covariance matrix $\bSigma^{\star}$ is set to the closest rank-$q$ approximation using singular value decomposition of the Toeplitz matrix $(t_{jk})$, $j,k =1, \ldots, p$, defined as
\begin{equation*}
t_{jk} = 3 \times \mathds{1}_{j = k} + u^{\lvert j-k \rvert},
\end{equation*}
where $u$ is drawn uniformly between $0.6$ and $0.8$.

\paragraph{Experimental design} Each algorithm runs for a total of 1010
epochs with two regimes of batch sizes, $\mathfrak{B} = n$ and $\mathfrak{B} =1$, used along iterations.
We resort to the full dataset for the first 1000 epochs.
During this initial phase, the \BB{Monte Carlo gradient estimator \citep{mohamed2020montecarlogradientestimation}} is computed as the average of SNIS
estimators \eqref{eqn:is} over all individuals, namely
\begin{equation*}
\widehat{g}^{(t)} = -\frac{1}{n} \sum_{i=1}^n s_i^N\cp(\btheta^{(t)}\cp).
\end{equation*}
\BB{
Given the computed gradient estimator, parameters are updated using the
\textsf{Rprop} \citep{Rprop1993Riedmiller}} update rules : \textsf{Rprop}
assigns an individual learning rate to each parameter, which is adjusted based
on the gradient.
Formally, given $A^{(1)} = \mathbf I_d$ and $(\eta_+, \eta_-) = (1.2,0.5)$, \textsf{Rprop} updates correspond to
\begin{equation}
\btheta^{(t+1)} = \btheta^{(t)} - A^{(t)} \widehat{g}^{(t)}, \label{eq:rprop}
\end{equation}
with $A^{(t)}_{jk} = 0$ for $j \neq k$, $j,k = 1,\dots,d$ and
\begin{align*}
A^{(t)}_{jk} & =
\begin{cases} \eta_{+}A^{(t-1)}_{kk} &\text{if } \widehat{g}^{(t)}_k \text{ has same sign than } \widehat{g}^{(t-1)}_{k} \nonumber
\\
\eta_{-} A^{(t-1)}_{kk} & \text{else}
\end{cases},
\quad k = 1,\dots,d.
\end{align*}
\BB{
\textsf{Rprop} typically converges faster than standard SGD because it adapts the step
size of each parameter individually based on the local geometry of the
log-likelihood. \textsf{Rprop} allows to make rapid progress in smooth
directions while reducing oscillations. Importantly, \textsf{Rprop} is more robust to poorly scaled gradients and
less sensitive to learning rate tuning. However, it is not well-suited for
highly stochastic settings: it assumes that gradient signs are stable and
meaningful across steps, which is typically the case only when the gradient is
computed over the full dataset.}

% We refer the reader to Appendix \ref{sec:app-rprop} for additional details on the use of
% \textsf{Rprop}.

This practical approach is motivated by the empirical observations that
\textsf{Rprop} updates tend to converge much faster than the original updates
using a single learning rate and could serve as a warm-up phase.
Our study focuses on the last 10 epochs with a batch size of 1.
This setting corresponds to $T = 3000$ iterations of Algorithm \ref{algo:sgis}.
The learning rate $\gamma$ is determined via a grid search.
We present results corresponding to $N = 500, 1000$, or $2000$ samples for the SNIS estimator \eqref{eqn:is}.

\paragraph{Quality of the importance sampling proposal distribution}

The proposal distributions $\nu_i(\cdot\,;\btheta)$ of the four competitors are compared at initialization with the three following metrics:
\begin{itemize}
    \item the Kullback--Leibler divergence $\mathrm{KL}[\nu_i(\cdot\,;\btheta) \Vert p_{\btheta}\left( \cdot \mid
        \bY_i\right)]$, which gives a discrepancy measure in terms of the
        variational objective function, and thus provides a comparison between
        the proposal distribution and the variational distribution. Given a
        $N$-sample $\bv_{i, 1}, \ldots, \bv_{i, N}$ from
        $\nu_i(\cdot\,;\btheta)$, its Monte Carlo estimator is
        \begin{equation*}
        -\log(N) - \frac{1}{N}\sum_{r = 1}^N \log (\omega_{\btheta, i, r}),
        \quad
        \omega_{\btheta, i, r}= \frac{\rho_{\btheta, i}(\bv_{i, r})}{\sum_{s = 1}^N \rho_{\btheta, i}(\bv_{i, s})}.
        \end{equation*}
    \item the Kullback--Leibler divergence $\mathrm{KL}[p_{\btheta}\left(\cdot \mid
        \bY_i\right) \Vert \nu_i(\cdot\,;\btheta)]$, which relates to the
        efficiency of the importance sampling scheme in terms of bias and
        quadratic error. Its Monte Carlo estimator is
        \begin{equation*}
        \log(N) + \sum_{r = 1}^N \omega_{\btheta, i, r} \log(\omega_{\btheta, i, r}).
        \end{equation*}
    \item The Effective Sample Size (ESS), which assesses how accu\-ra\-te\-ly the weigh\-ted samples from the importance sampling method approximates the target distribution $p_{\btheta}\left( \cdot \mid \bY_i\right)$: a higher effective sample size indicates a better empirical approximation of the target distribution. It is estimated by
    \begin{equation*}
    \left(\sum_{r = 1}^N \omega_{\btheta, i, r}^2\right)^{-1}.
    \end{equation*}
\end{itemize}

\begin{figure}[!ht]
    \begin{center}
        \includegraphics[width=\linewidth]{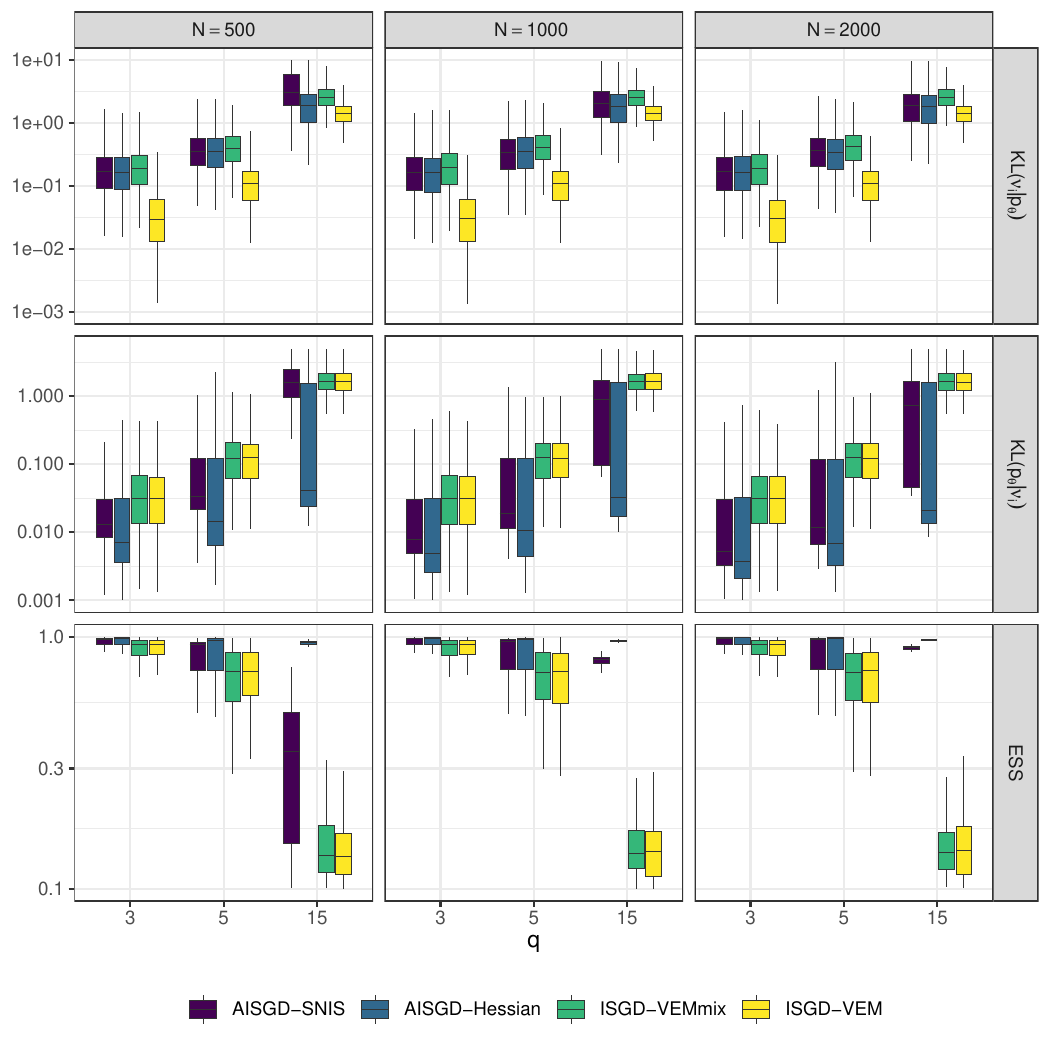}
        \caption{Distribution of the Kullback--Leibler divergence
        $\mathrm{KL}[\nu_i(\cdot\,;\btheta) \Vert p_{\btheta}\left( \cdot \mid
    \bY_i\right)]$ (top row), the Kullback--Leibler divergence
$\mathrm{KL}[p_{\btheta}\left(\cdot \mid  \bY_i\right) \Vert
\nu_i(\cdot\,;\btheta)]$ (middle row), and the effective sample size (bottom
row) as a function of the number of Monte Carlo draws $N$, the rank constraint
$q$, and the inference algorithms (AISGD-SNIS, AISGD-Hessian, ISGD-VEMmix,
ISGD-VEM) at initialization, that is $\btheta^{(1)}$ is the VEM estimate. Each boxplot is based on $M = 100$ synthetic
datasets, with each metric estimated using the specified $N$ Monte Carlo
draws.}
    \label{fig:quality}
    \end{center}
\end{figure}

The results are displayed in Figure \ref{fig:quality}.
A key distinction between the ISGD-VEM proposal and other methods is the use of
a Gaussian mixture distribution. A comparison between ISGD-VEM and ISGD-VEMmix
highlights the impact of adding a defensive component to the variational
solution. As shown in the top row of Figure \ref{fig:quality}, the defensive
component does not yield improvement in terms of the variational objective
function.  However, the ISGD-VEMmix proposal offers equivalent performances
regarding the importance sampling method (middle and bottom rows of
Figure~\ref{fig:quality}).

Furthermore, the comparison of ISGD-VEMmix with the two AISGD variants (SNIS
and Hessian) enables the evaluation of the impact of re\-pla\-cing the diagonal
covariance from the variational distribution with either the conditional
covariance estimate or the curvature estimate. Unsurprisingly, AISGD-SNIS
becomes numerically un\-sta\-ble when we do not have enough Monte Carlo draws
for a given rank constraint $q$.
It is well known that estimating a covariance
matrix is much more de\-man\-ding when the dimension of the sample space
increases. From the variational point of view (top row of Figure
\ref{fig:quality}), AISGD-SNIS and AISGD-Hessian do not exhibit significant
improvement compared to ISGD-VEMmix.
However, providing sufficient Monte Carlo draws are available, both AISGD-SNIS
and AISGD-Hessian consistently demonstrate superior performance for the
importance sampling scheme in comparison to variational-based proposals (middle
and bottom row of Figure \ref{algo:sgis}).
The lower values for the $\mathrm{KL}[p_{\btheta}\left(\cdot \mid \bY_i\right) \Vert \nu_i(\cdot\,;\btheta)]$
indicate a better control over bias and variance, while the performances in terms of ESS demonstrates that
AISGD solutions are robust and efficient when the dimension $q$ increases.
Overall, AISGD-Hessian offers the most favorable practical performances.
In what follows, we thus focus solely on AISGD-Hessian.

\paragraph{Asymptotic normality of the regression coefficient}
A key property of the MLE is the asymptotic
normality that provides tests and confidence intervals on the model parameters.
This property is not guaranteed by the variational estimator, maximizing a \textit{surrogate}
log-likelihood (often called ELBO).
In order to assess the ability of our method to provide valid tests and confidence intervals for each, say, regression parameter $B_{kj}$, we examine the standardized estimates
\begin{equation}
    \widetilde{B}_{kj} = \left(\widehat{B}_{kj} - B^{\star}_{kj}\right) \bigg/ \sqrt{\widehat{\mathbb{V}}[\widehat{B}_{kj}]}
    \label{eqn:stat-test}
\end{equation}
where $B^{\star}_{kj}$ denotes the true value, and
$\widehat{\mathbb{V}}[\widehat{B}_{kj}]$ stands for the estimated variance of $\widehat{B}_{kj}$.
Recall that the Fisher information matrix (FIM) is defined as
\begin{equation}
    \label{eqn:FIM}
    \mathcal{I}({\btheta}) = \mathbb{E}_{\btheta}\left[
    \nabla_{\btheta} \log p_{\btheta}(\bY_1) \left\lbrace\nabla_{\btheta} \log p_{\btheta}(\bY_1) \right\rbrace^{\top}
    \right].
\end{equation}
According to the M-estimator theory, the (asymptotic) variance of $\widehat{B}_{kj}$ is given by the diagonal term corresponding to $B_{kj}$ of the inverse of the FIM, and the distribution of the $\widetilde{B}^{(m)}_{kj}$ across simulations $m = 1,\ldots, M$ should be close to a standard normal.

We illustrate the fit of the test statistics, as defined in \eqref{eqn:stat-test}, to the standard normal distribution at three different stages of the estimation procedure.
The first stage corresponds to the VEM initialization, that is $\widehat{B}_{kj} = \widehat{B}^{\vem}_{kj}$,
and a variational proxy for the unknown (asymptotic) variance of the variational estimator.
This solution assumes that the evidence lower bound can be used in place of the log-likelihood in the definition \eqref{eqn:FIM}.
We refer to this solution as variational FIM.
The second stage corresponds to the first iteration of Algorithm \ref{algo:sgis-adaptive}.
The test statistics are also computed for $\widehat{B}_{kj} = \widehat{B}^{\vem}_{kj}$, but instead of using the variational FIM, we compute a Monte Carlo
estimator of the FIM, referred to as SNIS FIM, given by
\begin{equation}
\label{eqn:SNISFIM}
\widehat{\mathcal{I}}(\btheta) = \frac{1}{n} \sum_{i = 1}^n  \widehat{s}_i^{N}(\btheta) \left\lbrace\widehat{s}_i^{N}(\btheta)\right\rbrace^{\top}.
\end{equation}
The third stage corresponds to the end of the optimization scheme, that is, test statistics computed for $\widehat{B}_{kj} = \widehat{B}^{(T)}_{kj}$. The FIM at that point is also estimated using the Monte Carlo estimator \eqref{eqn:SNISFIM}.

We first perform a Kolmogorov-Smirnov test for the three aforementioned stages with $q = 5$ and $N = 1000$. Figure \ref{fig:ks} represents Kolmogorov-Smirnov $p$-values associated with the $d \times p$ regression parameter $B_{kj}$.
The figure shows that the normality hypothesis is not rejected for the AISGD-Hessian. However, as we could expect, it is rejected for the VEM standardized estimates.

\begin{figure}[!t]
    \begin{center}
        \includegraphics[width=\linewidth]{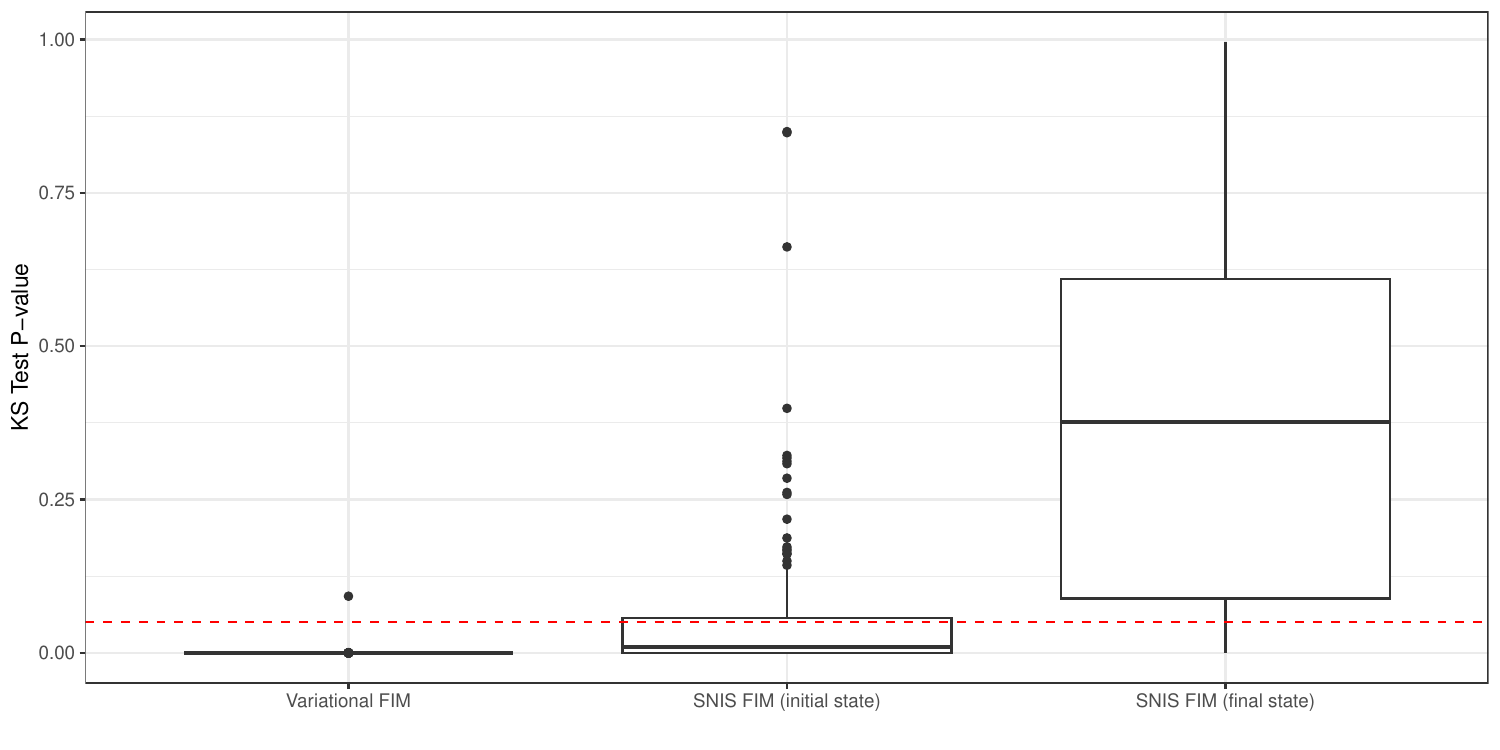}
        \caption{
Distribution of the $p$-values of the Kolmogorov-Smirnov test for the distribution of the standardized estimates $\widetilde{B}_{kj}$, as defined in \eqref{eqn:stat-test},
over $M = 100$ simulations ($n = 300$, $p = 150$, $q = 5$, and $d = 1$) at three different stages: initial state $\btheta^{(1)}$ (VEM estimate) with variational FIM (left) and with SNIS FIM (middle), and final state $\btheta^{(T)}$ of AISGD-Hessian with SNIS FIM (right).
AISGD-Hessian is run with $N = 1000$ Monte Carlo draws.
Each boxplot is built across the $d\times p$ normalized coefficients $\widetilde{B}_{kj}$.
Red dashed line {$\rm [\textcolor{red}{- -}]$}: the $\alpha = 5\%$ significance threshold.
        }
    \label{fig:ks}
    \end{center}
\end{figure}

Figure \ref{fig:qqplot} first illustrates why a departure from normality is
observed for the standardized variational estimates on certain regression
coefficients.

209: The variational FIM solution substantially underestimates the variance (left
210: column), suggesting that a fully variational approach would result in overly
211: narrow confidence intervals.
The variational FIM solution significantly underestimates the variance (left
column), implying, for instance, that using \BB{a fully variational approach would result in overly narrow confidence intervals.}
Similar results were reported on the PLN model in \cite{stoehr2024}.
Conversely, the SNIS FIM offers a satisfying and more accurate estimate of the variance (middle and right columns).
The VEM estimator also exhibits bias (middle column).
Further examination of additional qq-plots (not displayed here) indicates that VEM tends to overestimate the regression coefficients.
Finally, the right column of Figure \ref{fig:qqplot} shows that AISGD-Hessian corrects this bias and thus provides a good fit of the test statistics to the standard normal.

\paragraph{Conclusion}
AISGD-Hessian is a compelling method for inferring the parameters of a PLN PCA
model. The algorithm offers an estimation procedure with valid uncertainty
measures (or statistical tests). In contrast, the variational solution leads to
unreliable outputs due to the approximation used throughout the VEM scheme and
the crude variational FIM estimate.

\begin{figure}[!t]
    \begin{center}
        \includegraphics[width=\linewidth]{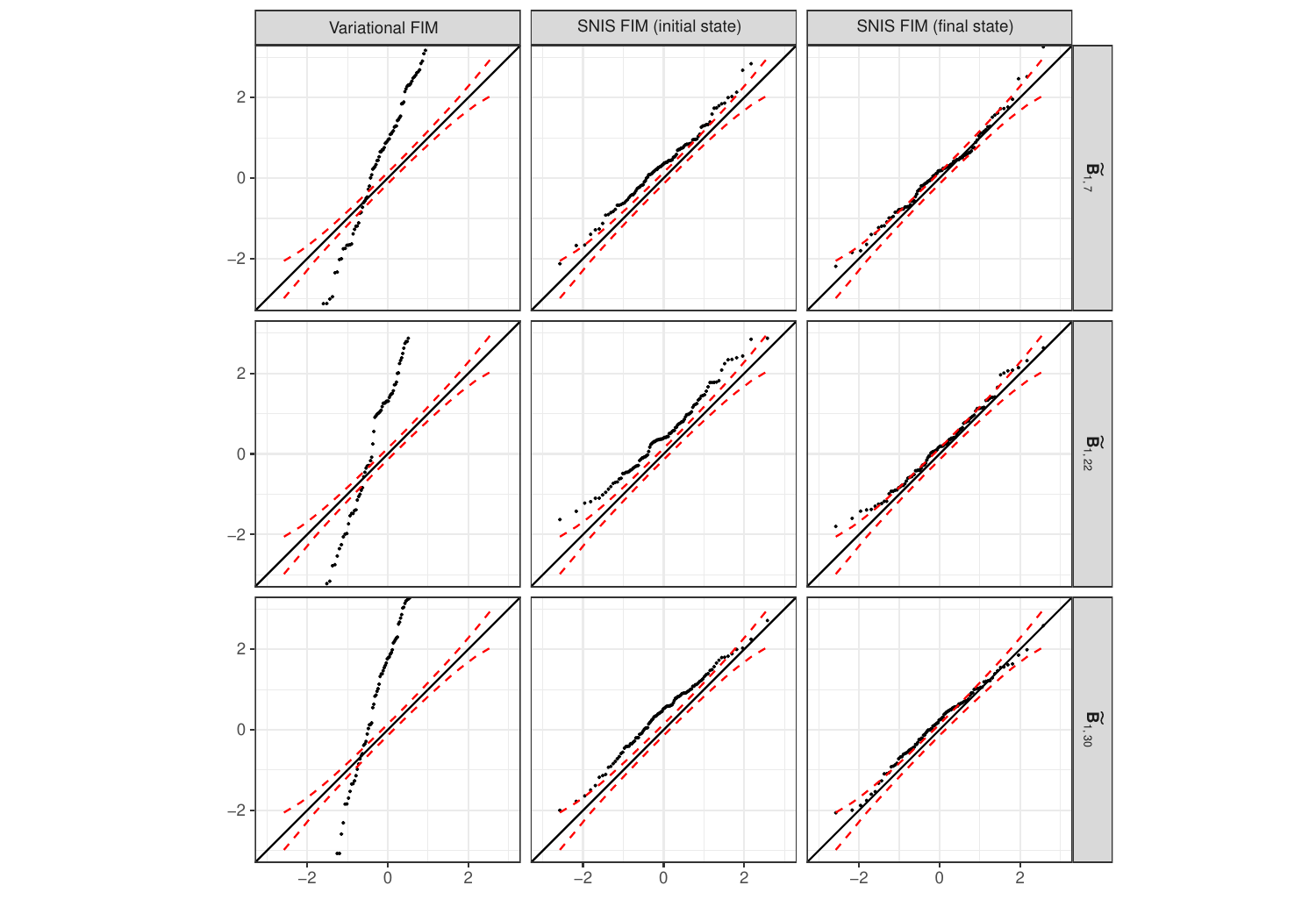}
        \caption{
qq-plots of the standardized regression coefficients $\widetilde{B}_{1j}$ ($j = 7, 22, 30$), as defined in \eqref{eqn:stat-test},
over $M = 100$ simulations ($n = 300$, $p = 150$, $q = 5$, and $d = 1$) at three different stages: initial state $\btheta^{(1)}$ (VEM estimate) with variational FIM (left) and with SNIS FIM (middle), and final state $\btheta^{(T)}$ of AISGD-Hessian with SNIS FIM (right). AISGD-Hessian is run with $N = 1000$ Monte Carlo draws.
$x$-axis: standard normal quantiles, $y$-axis: quantiles of $\widetilde{B}_{1j}$ (black dot $\rm [\bullet]$), red dashed lines  {$\rm [\textcolor{red}{- -}]$}: 95\% bounds for the standard normal qq-plot, black solid line  {$\rm [\textcolor{black}{-}]$}: perfect fit.
}
    \label{fig:qqplot}
    \end{center}
\end{figure}

\subsection{Inference on real data}
\label{sec:real}
\paragraph{Dataset} Real data analysis is based on the scMARK dataset \citep{scMark}.
The latter is a benchmark for single-cell Ribonucleic acid (scRNA) data
designed to serve as an RNA-seq equivalent of the MNIST dataset --- each cell
being labeled by one of the 28 possible cell types.
It corresponds to $n=19 998$ samples (cells) and $p = 14 059$ features (gene
expression).

Due to memory limitations, we cannot address the problem with the original $p =
14059$ features. Therefore, we perform inference on a dataset reduced to $n =
300$ randomly chosen samples from the two most prevalent cell types (\textit{T cells CD4+} and \textit{T cells CD8+}), focusing
on the $p = 100$ features with the largest variance.
The offset term $\bo$ is set to zero.
The two cell types are used as covariates, such that for each cell $i$, the covariates are given by
\begin{equation*}
x_{i1} = \begin{cases}
1 \quad  \text{if cell } i \text{ is a  T cells CD4+}
\\
0 \quad \text{else}
\end{cases},
\quad
x_{i2} = 1 - x_{i1}.
\end{equation*}

\paragraph*{Illustration of Theorem \ref{thm:convergence}}
For each rank constraint $q = 3, 5, 15$, we perform 10 runs of the AISGD-Hessian version of Algorithm \ref{algo:sgis-adaptive} initialized at $\btheta^{(1)} = \btheta^{\vem}$ for $100$ epochs and $N = 5000$ Monte
Carlo draws.
In Figure \ref{fig:real}, we monitor the negative marginal log-likelihood and the gradient norm, averaged over the 10 runs, over the iterations.
We observe a significant gain in terms of the marginal log-likelihood, suggesting that VEM has not converged to the maximum likelihood estimator. Interestingly, AISGD-Hessian proves to be increasingly more efficient as $q$ increases.
This behavior could result from using proposal distributions with a full covariance structure, which provides additional
insights on the underlying geometry as compared to the diagonal approximation used in VEM.
On the other hand, the norm of the gradient of the
objective function decreases in average with $T$ as stated by Theorem \ref{thm:convergence}.
\BB{In practice, one can monitor the variation in the norm of the parameters instead of the norm of the gradient to monitor the convergence of the algorithm to a critical point of the log-likelihood. As aforementionned, such criteria do not ensure the convergence to the MLE but solely to a critical point which is typical in non-convex optimization. A common strategy to look for better local maxima is to run the algorithm for various initializations.
}

\begin{figure}[!ht]
    \begin{center}
        \includegraphics[width=\linewidth]{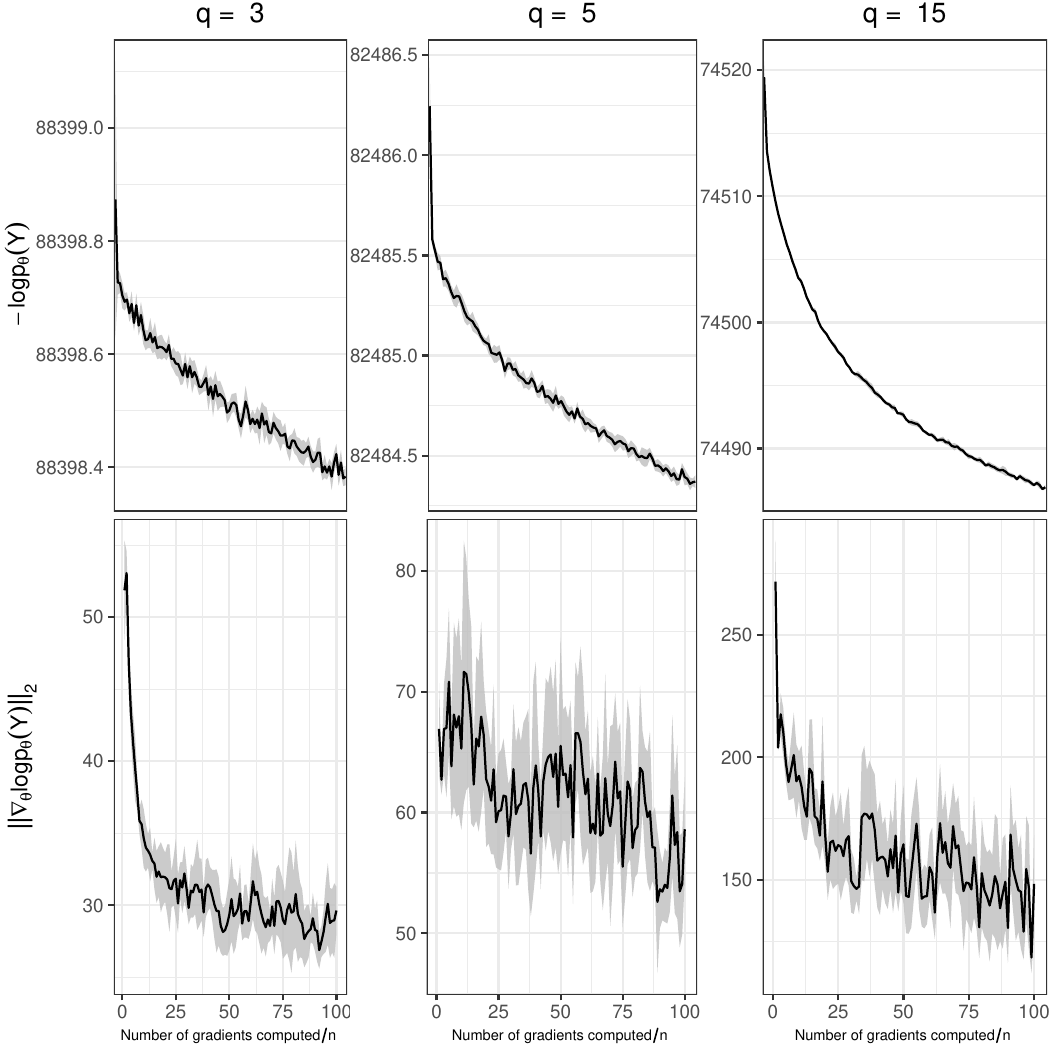}
        \caption{Negative marginal log-likelihood (top row) and norm of the gradient
        of the objective function (bottom row) for AISGD-Hessian as functions of the rank constraint $q$ and the number of epochs for the scMARK dataset reduced to $n = 300$ samples and $p = 100$ variables. Black solid line {$\rm [-]$}: averages over 10 runs of the AISGD-Hessian method initialized with the VEM estimate and $N = 5000$ Monte Carlo draws, grey area: 95\% confidence regions.
        }
    \label{fig:real}
    \end{center}
\end{figure}

\paragraph*{Parameters significance}
A major interest of our method is to achieve maximum likelihood estimation, but also to gain in interpretability thanks to an accurate estimation of the variance of PLN PCA parameters estimates.
This allows practitioners to perform statistical tests and determine if a covariate has a significant effect or not.
Figure \ref{fig:real_ic} presents the confidence intervals for regression coefficients obtained at the initial and final iterations of a run of the AISGD-Hessian method on the reduced scMARK dataset with $q = 5$. As detailed in Section \ref{sec:synthetic-data}, we used here a warm-up phase of $2000$ epochs, a batch size of $n = 300$ and $N = 1000$ Monte Carlo draws. We solely represent the coefficients whose estimates lie between $-1$ and $2$, leaving 34 coefficients out of the $200$ available.

\begin{figure}[!ht]
    \begin{center}
        \includegraphics[width=\linewidth]{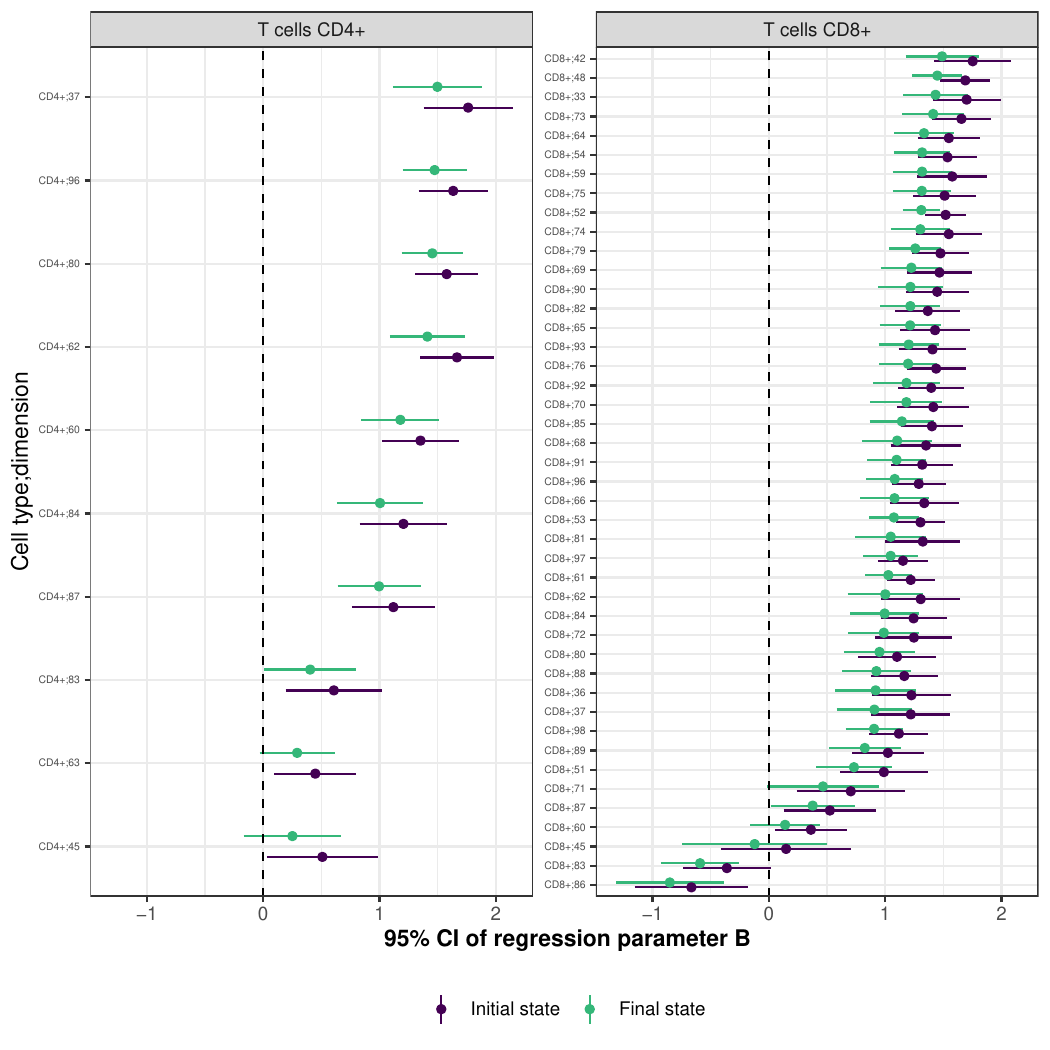}
        \caption{
            Confidence intervals for regression coefficients of $m = 2$ cell groups (CD8+
            and CD4+ T cells) from the scMARK dataset, reduced to $n = 300$ samples and $p
            = 100$ variables.
            The intervals are computed at the initial state corresponding to
            the VEM estimate (purple), and the final state of a run of
            AISGD-Hessian with $N=1000$ Monte Carlo draws. The variance was
            computed using the SNIS FIM estimator \eqref{eqn:SNISFIM}.
            $x$-axis: regression coefficient value, $y$-axis: cell type and feature number.
}
    \label{fig:real_ic}
    \end{center}
\end{figure}

As previously observed in Figure \ref{fig:qqplot}, the VEM method seems to overestimate the regression coefficients consistently.
By correcting this bias, the AISGD-Hessian method leads to differences in the interpretation for a few coefficients.
For instance, we could reject the null hypothesis $B_{kj} = 0$ for all CD4+
coefficients at the VEM initialization, implying a positive impact of all the
covariates. Conversely, at the final iteration, the decision changes for
CD4+;63, CD4+;45 which are no longer significant. The same observation stands
for CD8+;71 and CD8+;60. Additionally, CD8+;83, which initially has no impact,
is found to have a negative impact at the final iteration.

% DISCUSSION------------------------------------------------------------------------------------------
\section{Discussion}

The paper explores the opportunity of a projected stochastic gradient scheme for inferring the parameters of a PLN-PCA model.
Unlike competing variational approaches, our method allows to retrieve
statistical guarantees on the resulting estimate, representing a significant
step forward for the practitioners. This enables the construction of
uncertainty measures and statistical tests.
We have illustrated the benefits of our method on synthetic and real datasets using mixture distributions as proposals.
Obviously we could resort to any distribution provided it satisfies Assumption \ref{hyp:proposal}.  A future work could be to investigate the potential of normalizing flows as proposal distribution, especially for larger rank constraints.

While we have proven the convergence for the PLN-PCA model using an importance
sampling-based gradient estimator, the descent lemma \ref{lem:control} opens
new vista for latent variable models.  Indeed, its minimal assumptions offers
broader applicability, and it can serve as the corner stone for extending
Theorem \ref{thm:convergence} to models with $L$-smooth loss function and
arbitrary random and biased gradient estimator.

A natural extension is to consider models where the emission distribution
belongs to a natural (or canonical) exponential family.
While deriving regularity of the log likelihood, and hence a score estimator,
assessing the $L$-smoothness of the loss function is relatively
straightforward for some families --- typically when the natural parameter
space is the entire vector space ---, such as the Binomial distribution.
Additional technical conditions may however be required for families with
constrained natural parameter spaces, like the negative-Binomial or Gamma
distributions. These conditions might involve integrability constraints on the
moments of the emission distribution with respect to the latent distribution or
modifications to the link function $f_i$ to account for the parameter
constraints.

Beside addressing the inference for other models, the assumption of lemma
\ref{lem:control} on the gradient estimator allows to consider other, possibly
more elaborated, simulation based method to estimate the gradient, such as
diffusion models.

\BB{The present work focuses on a basic version of SGD, but it opens vista for future research,
particularly regarding adaptive variants of the method, in line with recent developments such as those proposed by \citet{surendran}.
A first attempt would be to investigate the use of a variable numbers of Monte Carlo samples $N$ along the iterations.
Indeed, in iterative schemes involving Monte Carlo approximations, high precision is generally not required in the early stages.
This suggests that  gradually increasing $N$ throughout the iterations --- an approach often advocated in the literature
\citep[\textit{e.g.,}][]{saem} --- may prove efficient in terms of computational cost.
Furthermore, we anticipate that the method could
benefit from dynamically adapting the learning
rate to the varying number of samples per
iteration to substantially improve the
convergence.
%Allowing for a variable number of iterations would enable the use of a smaller
%learning rate toward the end \citep{surendran}, while permitting a larger learning rate at the
%beginning, thus adapting the step size to the needs of the optimization
}

\section*{Fundings}
Bastien Bartardière and Julien Chiquet are supported by the French
ANR grant ANR-18-CE45-0023 Statistics and Machine Learning for
Single Cell Genomics (SingleStatOmics).
Julien Stoehr has been partly funded by the European Union (ERC-2022-SYG-OCEAN-101071601).
Views and opinions expressed are however those of the author only and do not necessarily reflect those of the European Union or the European Research Council Executive Agency.
Neither the European Union nor the granting authority can be held responsible for them.

%------------------------------------------------------------------------------------------
\bibliographystyle{imsart-nameyear}
\bibliography{bib}

% APPENDIX ------------------------------------------------------------------------------------------
\begin{appendix}

%------------------------------------------------------------------------------------------
%------------------------------------------------------------------------------------------
\section{Properties of the PLN-PCA model }%\eqref{eqn:PLNPCA} }
\label{sec:app-pln}
%------------------------------------------------------------------------------------------
%------------------------------------------------------------------------------------------

%------------------------------------------------------------------------------------------
%------------------------------------------------------------------------------------------
\begin{lemma}
\label{lem:unif-bound}
Under Model \eqref{eqn:PLNPCA}, for any compact set $\Theta \subset \mathbb{R}^d$, for all individual $i\in\{1, \ldots, n\}$, there exists two constants $K_i^{\Theta} > 0$ and $\kappa_i^{\Theta}\in\mathbb{R}$, such that for any $\btheta\in\Theta$
\begin{equation}
 \label{eq:linear_log_bounded}
        \log p_{\btheta} \left( \bY_i \mid \bW_i \right) \leq  K_i^{\Theta}\lVert \bW_i \rVert + \kappa_i^{\Theta}.
\end{equation}
\end{lemma}
%------------------------------------------------------------------------------------------
%------------------------------------------------------------------------------------------
\begin{proof}
Given $i\in\lbrace 1, \ldots, n\rbrace$, $\btheta = (\bB, \bC) \in\mathbb{R}^d$, due to the conditional independance,
we have for $\bZ_i = (Z_{ij})$ as defined by Equation \eqref{eqn:PLNPCA},
\begin{equation*}
\log p_{\btheta}(\bY_i\mid\bW_i) = \langle \bY_i, \bZ_i \rangle - \sum_{j = 1}^p %\lbrace
\exp(Z_{ij})
% + h(Y_{ij})\rbrace.
\end{equation*}
Since $\exp$ is convex and differentiable on $\mathbb{R}$, for any $z, z_0 \in \mathbb{R}$, we have
\begin{equation*}
\exp(z) \geq \exp(z_0) + (z - z_0)\exp(z_0).
\end{equation*}
It follows
\begin{equation*}
\log p_{\btheta}(\bY_i\mid\bW_i)
\leq \langle \bY_i, \bZ_i \rangle - \exp(z_0)\sum_{j = 1}^p Z_{ij} - \sum_{j = 1}^p
%\{h(Y_{ij}) +
(1 - z_0)\exp(z_0).
%\}.
\end{equation*}
Using the definition of $\bZ_i$ from Equation \eqref{eqn:PLNPCA}, we set
\begin{align*}
K_i(\btheta) & = \bC^\top\lbrace \bY_i - \exp(z_0)\mathbf{1}_p\rbrace,
\\
\kappa_i(\btheta) & =
\langle \bY_i - \exp(z_0)\mathbf{1}_p, \bB^\top\bx_i + \bo_i \rangle - \sum_{j = 1}^p
%\{h(Y_{ij}) +
(1 - z_0)\exp(z_0).
%\}.
\end{align*}
Then,
\begin{equation*}
    \log p_{\btheta}(\bY_i \mid \bW_i) \leq \langle K_i(\btheta), \bW_i\rangle + \kappa_i(\btheta),
\end{equation*}
and using the Cauchy--Schwarz inequality we get
\begin{equation*}
    \log p_{\btheta}(\bY_i \mid \bW_i)
    \leq \lVert K_i(\btheta) \rVert  \lVert \bW_i\rVert + \kappa_i(\btheta).
\end{equation*}
The functions $K_i$ and $\kappa_i$ are linear with respect to the parameter coordinates and, consequently, continuous on $\mathbb{R}^d$, and hence on $\Theta$.
Therefore, we can define the upper bounds $K_i^{\Theta} = \sup_{\btheta\in\Theta} \lVert K_i(\btheta) \rVert$ and $\kappa_i^{\Theta} = \sup_{\btheta\in\Theta} \kappa_i(\btheta)$, which provides the result.

\end{proof}
%------------------------------------------------------------------------------------------
%------------------------------------------------------------------------------------------

\sepline

%------------------------------------------------------------------------------------------
%------------------------------------------------------------------------------------------
\propscore*
%------------------------------------------------------------------------------------------
%------------------------------------------------------------------------------------------
\begin{proof}
The result is a direct application of the dominated convergence theorem. Indeed, for all individual $i = 1, \ldots, n$, the likelihood writes as
\begin{equation*}
p_{\btheta}(\bY_i) = \int_{\mathbb{R}^q} p_{\btheta}(\bY_i \mid \bw) \mathcal{N}(\mathrm{d}\bw; \mathbf{0}_q, \mathbf{I}_q).
\end{equation*}
The function $\btheta \mapsto p_{\btheta}(\bY_i \mid \bW_i)$ is twice continuously differentiable on $\mathbb{R}^d$ as a composition of such functions.
Moreover, each component $Z_{ij}$, $j = 1, \ldots, p$, is a linear function of the components of $\btheta$, and for $r, s = 1, \ldots, d$,
\begin{equation}
\label{app:eqn-partial}
\frac{\partial}{\partial \theta_r} Z_{ij} = \sum_{k = 1}^m x_{ik} \mathds{1}_{\theta_r = B_{kj}} + \sum_{k = 1}^q W_{ik} \mathds{1}_{\theta_r = C_{jk}},
\quad
\frac{\partial^2}{\partial \theta_s \partial \theta_r} Z_{ij} = 0.
\end{equation}
Consequently, given a component $\theta_r$, there is a unique index in $\{1, \ldots, p\}$ such that the first partial derivative is non-zero.
Denote by $j$ and $k$ such indices for the partial derivatives with respect to $\theta_r$ and $\theta_s$, respectively.
We then have
\begin{align*}
\frac{\partial}{\partial \theta_r} p_{\btheta}(\bY_i \mid \bW_i)
& = p_{\btheta}(\bY_i \mid \bW_i)
 \left\lbrace Y_{ij} - \exp(Z_{ij}) \right\rbrace \frac{\partial}{\partial \theta_r} Z_{ij},
\\
\frac{\partial^2}{\partial\theta_s \partial \theta_r} p_{\btheta}(\bY_i \mid \bW_i)
& = p_{\btheta}(\bY_i \mid \bW_i)
\Bigg[
    \BB{- \exp(Z_{ij}) \frac{\partial}{\partial \theta_r} Z_{ij} \frac{\partial}{\partial \theta_s} Z_{ij}}
\\
& \hspace{1.5em}
\BB{+\left\lbrace Y_{ij} - \exp(Z_{ij}) \right\rbrace \left\lbrace Y_{ik} - \exp(Z_{ik}) \right\rbrace \frac{\partial}{\partial \theta_r} Z_{ij} \frac{\partial}{\partial \theta_s} Z_{ik}}
\Bigg].
\end{align*}
It follows from \eqref{app:eqn-partial} that for any $r = 1, \ldots, d$, and any $j = 1, \ldots, p$
\begin{equation*}
\left\lvert \frac{\partial}{\partial \theta_r} Z_{ij} \right\rvert \leq \lVert \bx_i \rVert + \lVert \bW_i \rVert.
\end{equation*}
Given a non-empty open and bounded set $\Theta\subset\mathbb{R}^d$, it follows from Lemma \ref{lem:unif-bound} that it exists two constants $K_i^{\Theta} > 0$ and $\kappa_i^{\Theta}$ such that for any $\btheta\in\Theta$
\begin{equation*}
p_{\btheta}(\bY_i \mid \bW_i) \leq \exp \left( K_i^{\Theta} \lVert \bW_i \rVert  + \kappa_i^{\Theta} \right).
\end{equation*}
On the other hand, for any $\btheta\in\Theta$, any $j = 1, \ldots, p$
\begin{equation*}
\lvert Z_{ij} \rvert \leq A^{\Theta} \lVert \bW_i \rVert + R_i^{\Theta},
\quad\text{with}\quad
A^{\Theta} = \sup_{\btheta\in\Theta} \lVert \btheta \rVert,
\quad
R_i^{\Theta} = \lVert \bx_i \rVert \sup_{\btheta\in\Theta} \lVert \btheta \rVert +  \max_{j = 1,\ldots,p} \lvert o_{ij} \rvert.
\end{equation*}
Consequently, for any $j = 1, \ldots, p$, since $\exp(Z_{ij}) \leq \exp(\lvert Z_{ij} \rvert)$,
\begin{align*}
\lvert Y_{ij} - \exp(Z_{ij}) \rvert
& \leq \exp(\lvert Z_{ij} \rvert) \left\lbrace \lVert \bY_i \rVert \exp(-\lvert Z_{ij} \rvert) + 1 \right\rbrace \\
& \leq \BB{\exp\left(A^{\Theta} \lVert \bW_i \rVert + R_i^{\Theta}\right) \left( \lVert \bY_i \rVert + 1 \right).}
\end{align*}
Thereby, for any $j = 1, \ldots, p$,
\begin{align*}
\left\lvert
\left\lbrace Y_{ij} - \exp(Z_{ij}) \right\rbrace \frac{\partial}{\partial \theta_r} Z_{ij}
\right\rvert
& \leq \left( \lVert \bx_i \rVert + \lVert \bW_i \rVert \right)  \left(
 \left \lVert \bY_i \right \rVert  + 1
 \right) \exp\left(A^{\Theta} \lVert \bW_i \rVert + R_i^{\Theta}\right)
\\
\left\lvert
\exp(Z_{ij}) \frac{\partial}{\partial \theta_r} Z_{ij} \frac{\partial}{\partial \theta_s} Z_{ij}
\right\rvert
& \leq \left( \lVert \bx_i \rVert + \lVert \bW_i \rVert \right)^2 \exp\left(A^{\Theta} \lVert \bW_i \rVert + R_i^{\Theta}\right)
\\
& \leq \left( \lVert \bx_i \rVert + \lVert \bW_i \rVert \right)^2 \exp\left(2A^{\Theta} \lVert \bW_i \rVert + 2R_i^{\Theta}\right),
\end{align*}
where the last inequality follows because $A^{\Theta}$ and $R_i^{\Theta}$ are both positive.
Finally, we get that on $\Theta$
\begin{align*}
 \left\lvert
\frac{\partial}{\partial \theta_r} p_{\btheta}({\bY_i \mid \bW_i})
 \right\rvert
  \exp \left( -\frac{\lVert\bW_i \rVert ^2}{2} \right)
& \leq \left( \lVert \bx_i \rVert + \lVert \bW_i \rVert \right)
Q_i^{\Theta}(\bW_i, 1)
\\
 \left\lvert
\frac{\partial}{\partial \theta_s\partial \theta_r} p_{\btheta}({\bY_i \mid \bW_i})
 \right\rvert
  \exp \left( -\frac{\lVert\bW_i \rVert ^2}{2} \right)
& \leq \left( \lVert \bx_i \rVert + \lVert \bW_i \rVert \right)^2
 Q_i^{\Theta}(\bW_i, 2),
\end{align*}
with
\begin{align*}
    Q_i^{\Theta}(\bW_i, k)  =  \left(\left(\lVert \bY_i \rVert + 1
 \right)^k
+  k-1\right)&\exp\Bigg\lbrace
- \frac{1}{2}\lVert \bW_i \rVert^2 + \left(K_i^{\Theta} + k A^{\Theta}\right) \lVert \bW_i \rVert   \\
&
\BB{+ \kappa_i^{\Theta} + k R_i^{\Theta}}
\Bigg\rbrace.
\end{align*}
Crucially, each of these upper bounds is a Lebesgue integrable function on $\mathbb{R}^q$ that does not depend on $\btheta\in\Theta$.
Consequently, we can conclude that the likelihood is twice continuously differentiable on $\Theta$  with the dominated convergence theorem.
Since the result holds for any open and bounded set $\Theta\subset\mathbb{R}^d$, for all $\btheta\in\mathbb{R}^d$, we can apply it to an open $d$-ball with center $\btheta$.
Therefore, the likelihood is twice continuously differentiable on $\mathbb R^d$.
Moreover
\begin{equation*}
\nabla_{\btheta}p_{\btheta}(\bY_i)
= \int_{\mathbb{R}^q} \nabla_{\btheta}p_{\btheta}(\bY_i \mid \bw) \mathcal{N}(\mathrm{d}\bw; \mathbf{0}_q, \mathbf{I}_q))
=  \int_{\mathbb{R}^q} \nabla_{\btheta}p_{\btheta}(\bY_i, \bw)\mathrm{d}\bw.
\end{equation*}

For all $\bY_i\in\mathbb{N}^p$, the likelihood $\btheta\mapsto p_{\btheta}(\bY_i)$ is positive on $\mathbb{R}^d$.
Indeed, the integrand is positive everywhere since, by definition
\begin{equation*}
\bw \mapsto p_{\btheta}(\bY_i \mid \bw) = \prod_{j = 1}^p p(Y_{ij} \mid z_{ij}),
\quad (z_{i1}, \ldots, z_{ip})^{\top} = f_{i}(\bw;\bB, \bC),
\end{equation*}
and for any $z \in\mathbb{R}$, $y \mapsto p(y\mid z)$ is positive on $\mathbb{N}$.
The continuous differentiability of the log-likelihood follows directly by composition. Finally,
\begin{align*}
\nabla_{\btheta}\log p_{\btheta}(\bY_i)
= \frac{1}{p_{\btheta}(\bY_i)} \int_{\mathbb{R}^q} \nabla_{\btheta}p_{\btheta}(\bY_i, \bw)\mathrm{d}\bw
& = \int_{\mathbb{R}^q}
\frac{p_{\btheta}(\bY_i, \bw)
\nabla_{\btheta}\log p_{\btheta}(\bY_i, \bw)}{p_{\btheta}(\bY_i)}
\mathrm{d}\bw
\\
& = \int_{\mathbb{R}^q}
\nabla_{\btheta}\log p_{\btheta}(\bY_i, \bw)
p_{\btheta}(\mathrm{d}\bw \mid \bY_i).
\end{align*}

%\begin{align*}
%\nabla_{\bB} p_{\btheta}(\bY_i \mid \bW)
%& = p_{\btheta}(\bY_i \mid \bW) \bx_i \left\lbrace \bY_i - A'_{\odot}(\bC\bW + \bB^\top\bx_i + \bo_i) \right\rbrace^\top,
%\\
%\nabla_{\bC} p_{\btheta}(\bY_i \mid \bW)
%& = p_{\btheta}(\bY_i \mid \bW)
% \left\lbrace \bY_i -  A'_{\odot}(\bC\bW + \bB^\top\bx_i + \bo_i) \right\rbrace \bW^\top.
%\end{align*}
%Therefore, using successively the triangle inequality and the sub-multiplicativity of the Frobenius-norm, we get that on $\Theta$
%\begin{align*}
% \left \lVert
%\nabla_{\bB} p_{\btheta}(\bY_i \mid \bW)
% \right \rVert _{F}
%& \leq h_i(\bW)
% \left \lVert \bx_i \right \rVert
% \left\lbrace
% \left \lVert \bY_i \right \rVert  +  \left \lVert A'_{\odot}(\bC\bW + \bB^\top\bx_i + \bo_i)  \right \rVert
% \right\rbrace,
%\\
% \left \lVert
%\nabla_{\bC} p_{\btheta}(\bY_i \mid \bW)
% \right \rVert _{F}
%& \leq h_i(\bW)
% \left \lVert \bW \right \rVert
% \left\lbrace
% \left \lVert \bY_i \right \rVert  +  \left \lVert A'_{\odot}(\bC\bW + \bB^\top\bx_i + \bo_i)  \right \rVert
% \right\rbrace.
%\end{align*}
%In order to conclude with the dominated convergence theorem, we need to derive upper bounds that are independent of $\btheta$, and $\mathcal{N}(\mathbf{0}_q, \mathbf{I}_q)$-integrable.
%The assumption on $A'_{\odot}$ yields that
%\begin{equation*}
%\exp \left( -\frac{1}{2}\lVert\bW \rVert ^2 \right)
% \left \lVert A'_{\odot}(\bC\bW + \bB^\top\bx_i + \bo_i)  \right \rVert
%\leq \exp \left\lbrace- \left( \frac{1}{2} - \varepsilon \right)\lVert\bW \rVert ^2 \right\rbrace \Psi_i(\bW).
%\end{equation*}
\end{proof}
%------------------------------------------------------------------------------------------
%------------------------------------------------------------------------------------------

\sepline

%------------------------------------------------------------------------------------------
%------------------------------------------------------------------------------------------
\propsmoothness*
%------------------------------------------------------------------------------------------
%------------------------------------------------------------------------------------------
\begin{proof}
Let $\Theta\subset\mathbb{R}^d$ be a nonempty compact set.
\begin{enumerate}
\item From Proposition \ref{prop:score}, the likelihood of the observed data $\btheta \mapsto p_{\btheta}(\bY_i)$ is continuous on the compact set $\Theta$.
Thus, it is bounded and attains its bounds on $\Theta$.
Moreover, the likelihood of the observed data is positive on $\mathbb{R}^d$, so is its lower bound on $\Theta$.

\item A direct consequence of Proposition \ref{prop:score}, is that $\btheta \mapsto \nabla \ell(\btheta)$ is continuous differentiable on the compact set $\Theta$.
Therefore, it is Lipschitz continuous.
\end{enumerate}

\end{proof}
%------------------------------------------------------------------------------------------
%------------------------------------------------------------------------------------------

\sepline

%------------------------------------------------------------------------------------------
%------------------------------------------------------------------------------------------
\section{Intermediate optimization results and proof of Lemma \ref{lem:control}}
\label{sec:app-optim}
%------------------------------------------------------------------------------------------
%------------------------------------------------------------------------------------------

%We denote $\mathbb E_t$ the expectation $\mathbb E \left[ \cdot | \tat,... \ta^{(1)} \right]$ and $\lVert \cdot \rVert \triangleq \lVert \cdot \rVert_2$.
In this Section, we consider results for minimizing a function $\ell$ on $\mathbb{R}^d$. We denote $\Theta\subset\mathbb{R}^d$ a compact and convex set.

%------------------------------------------------------------------------------------------
%------------------------------------------------------------------------------------------
\begin{lemma}[Non-expansiveness, {\citet[][Prop 2.1.3]{Bertsekas}}]
\label{lemma:proj}
The projection $P_{\Theta}$   satisfies the non-expansiveness property, namely for all $\btheta, \btheta' \in\mathbb{R}^d$
\begin{equation*}
 \left\lVert
P_{\Theta}(\btheta) - P_{\Theta}(\btheta')
 \right\rVert
\leq
 \left\lVert
\btheta - \btheta'
 \right\rVert.
\end{equation*}
\end{lemma}
%------------------------------------------------------------------------------------------
%------------------------------------------------------------------------------------------

\sepline

%------------------------------------------------------------------------------------------
%------------------------------------------------------------------------------------------
For a real $\eta > 0$, the proximal operator of $F$ as defined in Equation \eqref{eqn:F} is given by
\begin{equation*}
\mathrm{prox}_{F}^{\eta}(\btheta) = \argmin_{\btheta' \in \mathbb{R}^d} \left\{ F(\btheta') + \frac {1}{2 \eta} \lVert \btheta - \btheta' \rVert^2 \right\}.
\end{equation*}

\begin{lemma}[{\citet[][Lemma 4.3]{smoothmap}}]\label{lem:grad_E}
If the function $\ell$ is $L$-smooth, for any real constant $\eta < 1/L$, the
Moreau envelope $F_{\eta}$ is differentiable on $\mathbb{R}^d$ and relates to
the proximal operator through the following equation
\begin{equation*}
    \nabla F_{\eta}(\btheta) = \frac{1}{\eta} \left\lbrace \btheta - \mathrm{prox}_{F}^{\eta}(\btheta) \right\rbrace.
\end{equation*}
\end{lemma}
%------------------------------------------------------------------------------------------
%------------------------------------------------------------------------------------------
\begin{lemma}[{\citet[][Theorem 4.5]{smoothmap}}]
\label{lem:Moreau_mapping}
If the function $\ell$ is $L$-smooth, for any real constant $\eta > 0$, the gradient mapping $G_{\eta}^{(t)}$, as defined in Equation \eqref{eq:gradient-mapping},  satisfy the following inequality
\begin{equation*}
 \cp\lVert G_{\eta}^{(t)} \cp\rVert
 \leq \left( 1 + \frac{L\eta}{L\eta+1} \right) \left( 1 + \sqrt{\frac{L\eta}{L\eta + 1}}  \right)
 \cp\lVert \nabla F_{\frac{\eta}{L\eta + 1}}\cp( \btheta^{(t)} \cp) \cp\rVert.
\end{equation*}
\end{lemma}
%------------------------------------------------------------------------------------------
%------------------------------------------------------------------------------------------

\sepline

%------------------------------------------------------------------------------------------
%------------------------------------------------------------------------------------------
A differentiable function $f:\mathbb R^d \mapsto \mathbb R$ is said to be
$\mu$-strongly-convex, $\mu\in\mathbb{R}_+^*$, if for all $\btheta, \btheta'\in\mathbb{R}^d$,
\begin{equation}\label{eq:stronglyconvex}
 f(\btheta') \geq f(\btheta) + \left\langle \nabla_{\btheta} f(\btheta), \btheta' - \btheta \right\rangle +
 \frac{\mu}{2} \lVert \btheta - \btheta' \rVert^2.
\end{equation}
%------------------------------------------------------------------------------------------
%------------------------------------------------------------------------------------------
\begin{lemma}
\label{lem:strongly}
If the function $\ell$ is $L$-smooth, then, for any real constant $\eta \in (0, 1/L)$, and any $\bar{\btheta}\in\mathbb{R}^d$, the following function is $(\eta^{-1} -L)$-strongly convex
\begin{equation*}
f : \btheta \mapsto \ell(\btheta) + \frac{1}{2\eta}\lVert \btheta - \bar{\btheta}\rVert^2.
\end{equation*}
\end{lemma}
%------------------------------------------------------------------------------------------
%------------------------------------------------------------------------------------------
\begin{proof}
Let $\btheta,\btheta' \in \mathbb{R}^d$. We have
\begin{equation}
\label{eqn:dev-f}
f(\btheta) + \left\langle \nabla f(\btheta), \btheta' - \btheta \right\rangle
=
\ell(\btheta) + \frac{1}{2 \eta} \lVert \btheta - \bar\btheta \rVert^2 + \left\langle \nabla \ell(\btheta) + \eta^{-1}(\btheta - \bar{\btheta}), \btheta' - \btheta \right\rangle.
\end{equation}
According to \citet[][Lemma 3.4]{bubeck2015convex}, since $\ell$ is $L-$smooth, it satisfies for all $\btheta,\btheta' \in \mathbb R^d$
\begin{equation}
\label{eq:sandwich}
 \left\lvert \ell(\btheta') - \ell(\btheta) - \left\langle \nabla \ell(\btheta), \btheta'-\btheta \right\rangle
 \right\rvert
\leq
\frac {L}{2} \lVert \btheta' - \btheta \rVert^2.
\end{equation}
It directly follows that
\begin{equation*}
\ell(\btheta) + \langle \nabla \ell(\btheta), \btheta' - \btheta \rangle - \frac{L}{2} \lVert \btheta - \btheta' \rVert^2 \leq \ell(\btheta').
\end{equation*}
Combining the latter inequality with Equation \eqref{eqn:dev-f} yields
\begin{dmath*}
f(\btheta) + \left\langle \nabla f(\btheta), \btheta' - \btheta \right\rangle +
\frac{\eta^{-1} - L}{2} \lVert \btheta - \btheta' \rVert^2
\leq
\ell(\btheta') + \frac{1}{2\eta}
 \left(
\lVert \btheta - \bar\btheta \rVert^2 + 2\langle \btheta - \bar\btheta, \btheta' - \btheta \rangle
+ \lVert \btheta - \btheta \rVert^2
 \right)
\end{dmath*}
For a scalar product and its associated norm, the identity $\lVert a + b \rVert^2 = \lVert a \rVert^2 + 2\langle a, b \rangle + \lVert b \rVert^2$ gives
\begin{equation*}
f(\btheta) + \left\langle \nabla f(\btheta), \btheta' - \btheta \right\rangle +
\frac{\eta^{-1} - L}{2} \lVert \btheta - \btheta' \rVert^2
\leq \ell(\btheta') + \frac{1}{2\eta}\lVert \btheta' - \bar\btheta\rVert
= f(\btheta').
\end{equation*}
Since $\eta^{-1} + L > 0$, we can conlude that $f$ is $(\eta^{-1} + L )$-strongly convex.

\end{proof}
%------------------------------------------------------------------------------------------
%------------------------------------------------------------------------------------------

\sepline

%------------------------------------------------------------------------------------------
%------------------------------------------------------------------------------------------
\begin{lemma}
\label{lem:convenient_bartat}
If the function $\ell$ is $L$-smooth, then for any $\btheta \in \Theta$, and any real constant $\eta \in \big(0, (2\Gamma + L)^{-1}\big]$, with $\Gamma = \sup_{\btheta \in \Theta} \lVert \nabla_{\btheta} \ell(\btheta) \rVert$,
we have
\begin{equation*}
\lVert \btheta - \mathrm{prox}_{F}^{\eta}(\btheta) \rVert \leq 1.
\end{equation*}
\end{lemma}
%------------------------------------------------------------------------------------------
%------------------------------------------------------------------------------------------
\begin{proof}
Given $\btheta \in \Theta$, we show that any point at a distance more than one from $\btheta$ is not $\mathrm{prox}_{F}^{\eta}(\btheta)$. This implies that necessarily
\begin{equation*}
\lVert \btheta - \mathrm{prox}_{F}^{\eta}(\btheta) \rVert \leq 1.
\end{equation*}
Let $\btheta'\in\mathbb{R}^d$ be a point such that $\lVert \btheta - \btheta' \rVert > 1$.
According to \citet[][Lemma 3.4]{bubeck2015convex}, since $\ell$ is $L-$smooth, it satisfies for all $\btheta,\btheta' \in \mathbb R^d$
\begin{equation}
\label{eq:descent}
\ell(\btheta) \leq \ell(\btheta') + \langle \nabla \ell(\btheta'), \btheta - \btheta' \rangle + \frac{L}{2} \lVert \btheta' - \btheta \rVert^2.
\end{equation}
Let first assume that $\btheta' \in \Theta$.
Using successively the Cauchy--Schwarz inequality, the positive bound $\Gamma$ and the assumption $\lVert \btheta - \btheta' \rVert > 1$, we get
\begin{equation}
\label{eqn:cs-bound}
\langle \nabla \ell(\btheta'), \btheta - \btheta' \rangle \leq \lVert \nabla \ell(\btheta') \rVert \lVert \btheta - \btheta' \rVert
< \Gamma \lVert \btheta - \btheta' \rVert^2.
\end{equation}
It results from Equation \eqref{eq:descent} and Equation \eqref{eqn:cs-bound} that
\begin{equation*}
\ell(\btheta) < \ell(\btheta') + \left( \Gamma + \frac{L}{2} \right) \lVert \btheta' - \btheta \rVert^2.
\end{equation*}
For $\eta \leq (2\Gamma + L)^{-1}$, we have $(\Gamma + L/2) \leq 1/(2\eta)$. Moreover, the functions $F$ and $\ell$ coincide on $\Theta$, so that
\begin{equation*}
F(\btheta) < F(\btheta') + \frac{1}{2\eta} \lVert \btheta' - \btheta \rVert^2.
\end{equation*}
We can therefore conclude that $\btheta' \neq \mathrm{prox}_{F}^{\eta}(\btheta)$.
Conversely, if $\btheta' \notin \Theta$, by definition of $F$, we also have that $\btheta' \neq \mathrm{prox}_{F}^{\eta}(\btheta)$.

\end{proof}
%------------------------------------------------------------------------------------------
%------------------------------------------------------------------------------------------

\sepline

%------------------------------------------------------------------------------------------
%------------------------------------------------------------------------------------------
\lemcontrol*
%------------------------------------------------------------------------------------------
%------------------------------------------------------------------------------------------
\begin{proof}
Given $t\in\mathbb{N}^*$ and $\eta \in \big(0, \max\{2\Gamma + L, 2L\}^{-1}\big]$, we define the virtual iterates $\bar\btheta^{(t)}$ as
\begin{equation*}
\bar\btheta^{(t)} = \mathrm{prox}_{F}^{\eta}\cp( \btheta^{(t)} \cp)
= \argmin_{\btheta \in \mathbb{R}^d} \left\{
F(\btheta)+\frac{1}{2 \eta} \cp\lVert \btheta - \btheta^{(t)} \cp\rVert^2
 \right\} .
\end{equation*}
Note that $\bar\btheta^{(t)} \in \Theta$. Indeed, according to the definition of the Moreau envelope
\begin{equation*}
F \left( \bar\btheta^{(t)} \right)
= F_{\eta}\cp( \btheta^{(t)} \cp) - \frac{1}{2\eta} \left\lVert \bar\btheta^{(t)} - \btheta^{(t)} \right\rVert^2
\leq F\cp( \btheta^{(t)} \cp).
\end{equation*}
The upper bound is finite, since $\btheta^{(t)}$ as defined in \eqref{eqn:proj-sgd} is within the compact $\Theta$.
The operator $F$ is thus finite at $\bar\btheta^{(t)}$, and consequently
$\bar\btheta^{(t)} \in \Theta$ since $F$ takes infinite value outside of $\Theta$.

As the virtual iterates are within $\Theta$, the definition of the Moreau envelope yields
\begin{equation}
\label{eqn:env-update}
F_\eta\cp( \btheta^{(t+1)} \cp)
%= \ell \left( \bar\btheta^{(t+1)} \right) + \frac{1}{2 \eta} \left\lVert \btheta^{(t+1)} - \bar\btheta^{(t+1)} \right\rVert^2
\leq \ell \left( \bar\btheta^{(t)} \right) + \frac{1}{2 \eta} \left\lVert \btheta^{(t+1)} - \bar\btheta^{(t)} \right\lVert^2.
\end{equation}
Using the definition of $\btheta^{(t+1)}$ and the fact that $\bar\btheta^{(t)} \in \Theta$, we obtain from Lemma \ref{lemma:proj}
\begin{align*}
\frac{1}{2 \eta} \left\lVert \btheta^{(t+1)} - \bar\btheta^{(t)} \right\rVert^2
& = \frac{1}{2 \eta} \left\lVert P_{\Theta}\cp( \btheta^{(t)} - \gamma\widehat{g}^{(t)} \cp) - P_{\Theta} \left( \bar\btheta^{(t)} \right) \right\rVert^2
\\
& \leq
\frac{1}{2 \eta}  \left\lVert \btheta^{(t)} - \gamma\widehat{g}^{(t)} - \bar\btheta^{(t)} \right\rVert^2.
\end{align*}
The bilinearity of the scalar product gives
\begin{equation*}
 \left\lVert \btheta^{(t)} - \gamma\widehat{g}^{(t)} - \bar\btheta^{(t)} \right\rVert^2
= \left\lVert \btheta^{(t)} - \bar\btheta^{(t)} \right\rVert^2
+ 2 \gamma \left\langle \bar\btheta^{(t)} - \btheta^{(t)}, \widehat{g}^{(t)} \right\rangle
+ \gamma^2\cp\lVert \widehat{g}^{(t)} \cp\rVert^2,
\end{equation*}
and the inequality \eqref{eqn:env-update} becomes
\begin{align*}
F_\eta\cp( \btheta^{(t+1)} \cp)
& \leq \ell \left( \bar\btheta^{(t)} \right) + \frac{1}{2\eta}
 \left\lVert \btheta^{(t)} - \bar\btheta^{(t)} \right\rVert^2
+ \frac{\gamma}{\eta} \left\langle \bar\btheta^{(t)} - \btheta^{(t)}, \widehat{g}^{(t)} \right\rangle
+ \frac{\gamma^2}{2\eta} \cp\lVert \widehat{g}^{(t)} \cp\rVert^2
\\
& \leq F_{\eta}\cp( \btheta^{(t)} \cp)
+ \frac{\gamma}{\eta} \left\langle \bar\btheta^{(t)} - \btheta^{(t)}, \widehat{g}^{(t)} \right\rangle
+ \frac{\gamma^2}{2\eta} \cp\lVert \widehat{g}^{(t)} \cp\rVert^2.
\end{align*}
Therefore,
\begin{dmath*}
\mathbb{E} \left[
F_\eta\cp( \btheta^{(t+1)} \cp) \mid \btheta^{(t)}
 \right]
\leq F_{\eta}\cp( \btheta^{(t)} \cp)
+ \frac{\gamma}{\eta} \left\langle \bar\btheta^{(t)} - \btheta^{(t)}, \mathbb{E} \left[ \widehat{g}^{(t)} \mid \btheta^{(t)} \right] \right\rangle
+ \frac{\gamma^2}{2\eta} \mathbb{E} \left[ \cp\lVert \widehat{g}^{(t)} \cp\rVert^2 \mid \btheta^{(t)} \right].
\end{dmath*}
We now study more in detail the second and the third terms of the right-hand
side.

Regarding the second term, we have
\begin{equation*}
 \left\langle \bar\btheta^{(t)} - \btheta^{(t)}, \mathbb{E} \left[ \widehat{g}^{(t)} \mid \btheta^{(t)} \right] \right\rangle
=
 \left\langle \bar\btheta^{(t)} - \btheta^{(t)}, \nabla_{\btheta} \ell\cp( \btheta^{(t)} \cp) \right\rangle
+ \left\langle \bar\btheta^{(t)} - \btheta^{(t)}, b^{(t)} \right\rangle,
\end{equation*}
where $b^{(t)}$ stands for the bias of the gradient estimate, namely
\begin{equation*}
b^{(t)} = \mathbb{E} \left[ \widehat{g}^{(t)} \mid \btheta^{(t)} \right] - \nabla_{\btheta} \ell\cp( \btheta^{(t)} \cp).
\end{equation*}
By combining the Cauchy--Schwarz inequality with Lemma \ref{lem:convenient_bartat}, which applies since $\eta \leq (2\Gamma + L)^{-1}$, we get
\begin{equation*}
 \left\langle \bar\btheta^{(t)} - \btheta^{(t)}, b^{(t)} \right\rangle
\leq \left\lVert \bar\btheta^{(t)} - \btheta^{(t)} \right\rVert \cp\lVert b^{(t)} \cp\rVert
\leq \cp\lVert b^{(t)} \cp\rVert.
\end{equation*}
Since the function $\ell$ is $L$-smooth, we can use the inequality \eqref{eq:sandwich}, namely
\begin{equation*}
 \left\langle \bar\btheta^{(t)} - \btheta^{(t)},\nabla_{\btheta} \ell\cp( \btheta^{(t)} \cp) \right\rangle
\leq \ell \left( \bar\btheta^{(t)} \right) - \ell\cp( \btheta^{(t)} \cp)  + \frac{L}{2} \left\lVert \bar\btheta^{(t)} - \btheta^{(t)} \right\rVert^2.
\end{equation*}
%In conclusion, since $\ell$ and $F$ coincides on $\Theta$, and both $\bar\btheta^{(t)}, \btheta^{(t)}\in\Theta$,
In conclusion we obtain
\begin{equation}
\label{eqn:term-1}
 \left\langle \bar\btheta^{(t)} - \btheta^{(t)}, \mathbb{E} \left[ \widehat{g}^{(t)} \mid \btheta^{(t)} \right] \right\rangle
\leq
 \ell \left( \bar\btheta^{(t)} \right) - \ell\cp( \btheta^{(t)} \cp)  + \frac{L}{2} \left\lVert \bar\btheta^{(t)} - \btheta^{(t)} \right\rVert^2 + \cp\lVert b^{(t)} \cp\rVert.
\end{equation}
We aim at obtaining a bound not depending on the virtual iterates. As
$\eta^{-1} \geq 2L > L$, Lemma \ref{lem:strongly} applies, and the function
\begin{equation*}
g : \btheta \mapsto \ell(\btheta) + \frac{1}{2\eta} \cp\lVert \btheta - \btheta^{(t)}\cp\rVert^2
\end{equation*}
is $(\eta^{-1} - L)$-strongly convex. By definition, this function achieves a minimum at $\bar\btheta^{(t)}$. Consequently its gradient is 0 at that point and the strong-convexity gives
\begin{equation*}
\ell\cp( \btheta^{(t)} \cp) \geq \ell \left( \bar\btheta^{(t)} \right) + \frac{1}{2\eta} \left\lVert \bar\btheta^{(t)} - \btheta^{(t)} \right\rVert^2
+ \frac{\eta^{-1} - L}{2} \left\lVert \btheta^{(t)} - \bar\btheta^{(t)} \right\rVert^{2}.
\end{equation*}
Using this last inequality, Equation \eqref{eqn:term-1} becomes
\begin{align*}
 \left\langle \bar\btheta^{(t)} - \btheta^{(t)}, \mathbb{E} \left[ \widehat{g}^{(t)} \mid \btheta^{(t)} \right] \right\rangle
 \leq &
- \frac{1}{2\eta} \left\lVert \bar\btheta^{(t)} - \btheta^{(t)} \right\rVert^2 - \frac{\eta^{-1} -2L}{2} \left\lVert \bar\btheta^{(t)} - \btheta^{(t)} \right\rVert^2 \\
& + \cp\lVert b^{(t)}  \cp\rVert \\
 \leq &- \frac{1}{2\eta} \left\lVert \bar\btheta^{(t)} - \btheta^{(t)} \right\rVert^2 + \cp\lVert b^{(t)}\cp\rVert,
\end{align*}
since $\eta^{-1} \geq 2L$. Finally, using Lemma \ref{lem:grad_E} leads to
\begin{equation*}
\frac{1}{2 \eta} \left\lVert \bar\btheta^{(t)} - \btheta^{(t)} \right\rVert^2
=  \frac{1}{2 \eta} \cp\lVert \eta \nabla F_{\eta}\cp( \btheta^{(t)} \cp) \cp\rVert^2 = \frac{\eta}{2} \cp\lVert \nabla F_{\eta}\cp( \btheta^{(t)} \cp) \cp\rVert^2.
\end{equation*}
So far, we have hence proven that
\begin{dmath*}
\mathbb{E} \left[
F_\eta \cp( \btheta^{(t+1)} \cp) \mid \btheta^{(t)}
 \right]
\leq F_{\eta} \cp( \btheta^{(t)} \cp)
- \frac{\gamma}{2} \cp\lVert \nabla F_{\eta}\cp( \btheta^{(t)} \cp) \cp\rVert^2
+ \frac{\gamma}{\eta} \cp\lVert b^{(t)} \cp\rVert
+ \frac{\gamma^2}{2\eta} \mathbb{E} \left[ \cp\lVert \widehat{g}^{(t)} \cp\rVert^2 \mid \btheta^{(t)} \right].
\end{dmath*}
We now focus on the third term. The bilinearity of the scalar product gives
\begin{align*}
\mathbb{E} \left[ \cp\lVert \widehat{g}^{(t)} \cp\rVert^2 \mid \btheta^{(t)} \right]
& = \sigma^{(t)}
- \cp\lVert \nabla_{\btheta} \ell \cp( \btheta^{(t)} \cp) \cp\rVert^2
+ 2 \left\langle \mathbb{E} \left[ \widehat{g}^{(t)} \mid \btheta^{(t)} \right], \nabla_{\btheta} \ell \cp( \btheta^{(t)} \cp) \right\rangle
\\
& = \sigma^{(t)}
+ \cp\lVert \nabla_{\btheta} \ell \cp( \btheta^{(t)} \cp) \cp\rVert^2
+ 2 \cp\langle b^{(t)}, \nabla_{\btheta} \ell \cp( \btheta^{(t)} \cp) \cp\rangle.
\end{align*}
The Cauchy-Schwarz inequality along with the definition of the bound $\Gamma$ yields
\begin{equation*}
\mathbb{E} \left[ \cp\lVert \widehat{g}^{(t)} \cp\rVert^2 \mid \btheta^{(t)} \right]
\leq \sigma^{(t)} + \Gamma^2 + 2 \Gamma \cp\lVert b^{(t)} \cp\rVert.
\end{equation*}
It follows that
\begin{dmath*}
\mathbb{E} \left[
F_\eta \cp( \btheta^{(t+1)} \cp) \mid \btheta^{(t)}
 \right]
\leq F_{\eta} \cp( \btheta^{(t)} \cp)
- \frac{\gamma}{2} \cp\lVert \nabla F_{\eta}\cp( \btheta^{(t)} \cp) \cp\rVert^2
+ \frac{\gamma + \gamma^2\Gamma}{\eta} \cp\lVert b^{(t)} \cp\rVert
+ \frac{\gamma^2}{2\eta} \cp( \sigma^{(t)} + \Gamma^2 \cp).
\end{dmath*}

%\JScom{Now, using $\Esp{\lVert X \rVert^2} = \Esp{\lVert X  - \Esp{X}\rVert^2} + \lVert\Esp{ X }\rVert^2$  on the last expectation gives
%    \begin{align*}
%        \mathbb{E}_t \left[ \left\|\gt \right\|_2^2 \right]  &  = \mathbb{E}_t \left[ \left\|\gt - \Espt{\gt} \right\|_2^2 \right] + \lVert \Espt{\gt} \rVert^2 \\
%                                                         & = \mathbb{E}_t \left[ \left\|\gt - \Espt{\gt} \right\|_2^2 \right] + \lVert \nabla \ell \left( \tat \right) + \bt \rVert^2 \\
%                                                         & \leq \mathbb{E}_t \left[ \left\|\gt - \Espt{\gt} \right\|_2^2 \right] + 2  \lVert \nabla \ell \left( \tat \right)\rVert^2 + 2 \lVert \bt \rVert^2 \\
%                                                         & \leq \sigmat + 2 \Gamma^2 + 2 \xit
%    \end{align*}
% where we used $\lVert a + b\rVert^2 \leq 2 \lVert a \rVert^2 + 2 \lVert b \rVert^2$ at the penultimate and the last line is true by assumption. ON NE PEUT PAS CONCLURE AINSI CAR DANS $\sigma^{(t)}$, $\Espt{\gt} \neq \nabla \ell(\btheta^{(t)})$.}
%Putting everything together gives
%\begin{align*}
%\Espt{F_\eta \left( \tatp \right)} & \leq F_\eta \left( \tat \right)-\frac{\gamma}{2} \left\|\nabla F_\eta \left( \tat \right) \right\|_2^2 + \frac \gamma \eta \lVert \bt \rVert +\frac{\gamma^2 \left( \Gamma^2 +  \xit + \frac{\sigmat}2 \right)}{\eta}.
%\end{align*}
%Replacing $\lVert \bt \rVert$ by $\sqrt{\xit}$ finishes the proof.
\end{proof}

\sepline

%------------------------------------------------------------------------------------------
%------------------------------------------------------------------------------------------
\section{Proof of Theorem \ref{thm:convergence}}
\label{sec:app-convergence}
%------------------------------------------------------------------------------------------
%------------------------------------------------------------------------------------------

%------------------------------------------------------------------------------------------
%------------------------------------------------------------------------------------------
\propagapiou*
%------------------------------------------------------------------------------------------
%------------------------------------------------------------------------------------------
\begin{proof}
    For any $t\in\mathbb{N}$, since the individual $i(t)$ is drawn uniformly independently of $\btheta^{(t)}$, we have
\begin{align*}
\overline{\sigma}_{\rm IS}^{(t)}
& = \mathbb{E} \left[
\mathbb{E} \left[
 \left\lVert
\widehat{s}^{N}_{i(t)} - \nabla \log p_{\btheta^{(t)}} \left( \bY_{i(t)} \right)
 \right\rVert^2
\mid \btheta^{(t)}, i(t)
 \right]
 \right]
\\
& = \frac{1}{n}\sum_{i = 1}^n \mathbb{E} \left[
 \left\lVert
\widehat{s}^{N}_{i(t)} - \nabla \log p_{\btheta^{(t)}} \left( \bY_{i(t)} \right)
 \right\rVert^2
\mid \btheta^{(t)}, i(t) = i
 \right]
\\
& = \frac{1}{n}\sum_{i = 1}^n \mathbb{E}_{\nu^{\otimes N}_i(\cdot\,;\btheta^{(t)})} \left[
 \left\lVert
\widehat{s}^{N}_{i} - \nabla \log p_{\btheta^{(t)}} \left( \bY_{i} \right)
 \right\rVert^2
 \right]
\end{align*}
and, similarly
\begin{align*}
\overline{\xi}_{\rm IS}^{(t)}
& = \frac{1}{n}\sum_{i = 1}^n \left\lVert
\mathbb{E} \left[
\widehat{s}_{i(t)}^{N} -  \nabla \log p_{\btheta^{(t)}} \left( \bY_{i(t)} \right) \mid \btheta^{(t)}, i(t) = i
 \right]
 \right\rVert
\\
& = \frac{1}{n}\sum_{i = 1}^n  \left\lVert
\mathbb{E}_{\nu^{\otimes N}_i(\cdot\,;\btheta^{(t)})} \left[
\widehat{s}_{i}^{N} -  \nabla \log p_{\btheta^{(t)}} \left( \bY_{i} \right)
 \right]
 \right\rVert
\end{align*}
Consider an individual $i\in\{1, \ldots, n\}$ and a parameter $\btheta\in\Theta$.
As in \cite{agapiou2017}, the $r$-th central moment, $r\in\mathbb{N}^*$, of a mesurable function $h:\mathbb{R}^q \rightarrow\mathbb{R}$ with respect to the proposal distribution is denoted
\begin{equation*}
m_r[h] = \mathbb{E}_{\nu_i(\cdot\, ; \btheta)} \left[
 \left\lvert
h(\bV) - \mathbb{E}_{\nu_i(\cdot\, ; \btheta)} \left[ h(\bV) \right]
 \right\rvert^r
 \right].
\end{equation*}
We also denote $\phi_k$, $k\in\{1, \ldots, d\}$, the application that returns the $k$-th component of the score, namely for all $\bV\in\mathbb{R}^q$
\begin{equation*}
\phi_k(\bV) = \frac{\partial}{\partial \theta_k} \log p_{\btheta}(\bY_i, \bV).
\end{equation*}
For a $N$-sample $(\bV_1, \ldots, \bV_N)$ from $\nu_i(\cdot\, ; \btheta)$, the importance sampling estimate $\widehat{s}^N_i(\btheta) = (\widehat{s}_{ik})_{1\leq k \leq d}$ of the score $g_i$, as defined in Equation \eqref{eqn:score-marginal}, can thereby be written as
\begin{equation*}
 \widehat{s}_{ik} = \frac{1}{\sum_{s = 1}^N \rho_{\btheta, i}(\bV_s)}\sum_{r = 1}^N \rho_{\btheta, i}(\bV_r) \phi_k(\bV_r).
\end{equation*}
According to Theorem 2.3 in \cite{agapiou2017}, we can guarantee control over the bias and mean squared error of
the importance sampling estimate for the $k$-th component of the score, $k\in\{1, \ldots, d\}$, provided that the following quantity is finite
\begin{dmath*}
C_{\mathrm{MSE}, k}(\btheta)
= \frac{3}{p_{\btheta}(\bY_i)^2}
 \left\lbrace
m_2[\phi_k\rho_{\btheta, i}]
+ \frac{9}{p_{\btheta}(\bY_i)^2} \sqrt{
m_4[\rho_{\btheta, i}] \mathbb{E}_{\nu_i(\cdot\,;\btheta)} \left[ \left\lvert \phi_k(\bV)\rho_{\btheta, i}(\bV) \right\rvert^4 \right]
}
+ \frac{125}{p_{\btheta}(\bY_i)} \sqrt{
m_6[\rho_{\btheta, i}] \mathbb{E}_{\nu_i(\cdot\,;\btheta)} \left[ \left\lvert \phi_k(\bV) \right\rvert^4 \right]
}
 \right\rbrace.
\end{dmath*}
%where $C_4$ and $C_6$ are positive constants such that $C_4 \leq 3^4$ and $C_6 \leq 5^6$.
To prove such a statement, first note that by definition
\begin{equation}
\label{eqn:proof-cmse-1}
0 <  \lambda_i,
\qquad
\mathbb{E}_{\nu_i(\cdot\,;\btheta)} \left[ \left\lvert \phi_k(\bV) \right\rvert^4 \right] \leq \beta_i.
\end{equation}
The monotonicity of the expected value thus yields
\begin{equation}
\label{eqn:proof-cmse-2}
\mathbb{E}_{\nu_i(\cdot\,;\btheta)} \left[ \left\lvert \phi_k(\bV)\rho_{\btheta, i}(\bV) \right\rvert^4 \right]
\leq
\lambda_i^4 \mathbb{E}_{\nu_i(\cdot\,;\btheta)} \left[ \left\lvert \phi_k(\bV) \right\rvert^4 \right]
\leq \lambda_i^4\beta_i.
\end{equation}
Furthermore, we use that for all $r\in\mathbb{N}^*$, $\mathbb{E}[\lvert X - \mathbb{E}[X]\rvert^r] \leq 2^r\mathbb{E}[\lvert X \rvert^r]$ (this is a direct consequence of the Minkowski and Jensen inequalities).
Then, we have for all $r\in\mathbb{N}^*$
\begin{equation}
\label{eqn:proof-cmse-3}
\begin{aligned}
m_r[\rho_{\btheta, i}]
& \leq 2^r \mathbb{E}_{\nu_i(\cdot\, ; \btheta)} \left[
 \rho_{\btheta, i}(\bV)^r
 \right]
\leq 2^r\lambda_i^r,
\\
m_2[\phi_k\rho_{\btheta, i}]
& \leq 4\mathbb{E}_{\nu_i(\cdot\,;\btheta)} \left[ \left\lvert \phi_k(\bV)\rho_{\btheta, i}(\bV) \right\rvert^2 \right]
\leq
4\lambda_i^2 \mathbb{E}_{\nu_i(\cdot\,;\btheta)} \left[ \left\lvert \phi_k(\bV) \right\rvert^2 \right].
\end{aligned}
\end{equation}
Finally, the Jensen inequality for the square root function provides
that
\begin{equation}
\label{eqn:proof-cmse-4}
\mathbb{E}_{\nu_i(\cdot\,;\btheta)} \left[ \left\lvert \phi_k(\bV) \right\rvert^2 \right] \leq \sqrt{\mathbb{E}_{\nu_i(\cdot\,;\btheta)} \left[ \left\lvert \phi_k(\bV) \right\rvert^4 \right]} \leq \sqrt{\beta_i}
\end{equation}
Combining the upper-bounds from Equations \eqref{eqn:proof-cmse-1}--\eqref{eqn:proof-cmse-4} leads to
\begin{equation*}
C_{\mathrm{MSE}, k}(\btheta)
\leq \frac{12\lambda_i^2\sqrt{\beta}_i}{p_{\btheta}(\bY_i)^2}
 \left\lbrace
1 + \frac{9\lambda_i^2}{p_{\btheta}(\bY_i)^2} + \frac{250\lambda_i}{p_{\btheta}(\bY_i)}
 \right\rbrace.
\end{equation*}
Since $\lambda_i$ and $\beta_i$ are all finite constants, and for any $\btheta\in\mathbb{R}^d$ and any $\bY_i\in\mathbb{N}$, $p_{\btheta}(\bY_i)$ is positive, it follows that $C_{\mathrm{MSE}, k}(\btheta)$ is finite. Consequently,
 Theorem 2.3 in \cite{agapiou2017} states that
\begin{align*}
\mathbb{E}_{\nu^{\otimes N}_i(\cdot\, ; \btheta)} \left[
 \left\lbrace
\widehat{s}_{ik} - \frac{\partial}{\partial \theta_k} \log p_{\btheta}(\bY_i)
 \right\rbrace^2
 \right]
& \leq \frac{1}{N}C_{\mathrm{MSE}, k}(\btheta),
\\
 \left\lvert
\mathbb{E}_{\nu^{\otimes N}_i(\cdot\, ; \btheta)} \left[ \widehat{s}_{ik} \right]  - \frac{\partial}{\partial \theta_k} \log p_{\btheta}(\bY_i)
 \right\rvert
& \leq \frac{2}{Np_{\btheta}(\bY_i)}
\Bigg\lbrace
\frac{1}{p_{\btheta}(\bY_i)}\sqrt{m_2[\rho_{\btheta, i}] m_2[\overline{\phi_k} \rho_{\btheta, i}]}
\\
& \qquad\quad
+ \sqrt{C_{\mathrm{MSE}, k}(\btheta) \mathbb{E}_{\nu_i(\cdot\,;\btheta)} \left[ \rho_{\btheta, i}(\bV)^2 \right]}
\Bigg\rbrace,
\end{align*}
where $\overline{\phi_k} : \bV \mapsto \phi_k(\bV) - \mathbb{E}_{\nu_i(\cdot\,;\btheta)} \left[ \phi_k(\mathbf{U}) \right]$.
To conclude, we further need to eliminate the dependence in $\btheta$ in these upper bounds.
For any $\bY_i\in\mathbb{N}^p$, the function $\btheta \mapsto p_{\btheta}(\bY_i)$ is bounded below by $\zeta_i > 0$ on $\Theta$ (Proposition \ref{prop:smoothness}).
Moreover, combining the argument of Equation \eqref{eqn:proof-cmse-3} with the positivity of the variance and Equation \eqref{eqn:proof-cmse-4}, we have
\begin{equation*}
m_2[\overline{\phi_k} \rho_{\btheta, i}]
\leq 4\lambda_i^2 \mathbb{E}_{\nu_i(\cdot\,;\btheta)} \left[ \left\lvert \overline{\phi_k}(\bV) \right\rvert^2 \right]
\leq 4\lambda_i^2 \mathbb{E}_{\nu_i(\cdot\,;\btheta)} \left[ \left\lvert \phi_k(\bV) \right\rvert^2 \right]
\leq 4\lambda_i^2 \sqrt{\beta_i}.
\end{equation*}
Overall, these last two bounds in addition with Equations \eqref{eqn:proof-cmse-1}--\eqref{eqn:proof-cmse-4} yield the conclusion, namely
\begin{align*}
\overline{\sigma}_{\rm IS}^{(t)}
& = \frac{1}{n}\sum_{i=1}^n \sum_{k = 1}^d \mathbb{E}_{\nu^{\otimes N}_i(\cdot\, ; \btheta)} \left[
 \left\lbrace
\widehat{s}_{ik} - \frac{\partial}{\partial \theta_k} \log p_{\btheta}(\bY_i)
 \right\rbrace^2
 \right]
\\
& \leq
\frac{12d}{nN}\sum_{i = 1}^n \frac{\lambda_i^2\sqrt{\beta}_i}{\zeta_i^2}
 \left(
1 + \frac{250\lambda_i}{\zeta_i} + \frac{9\lambda_i^2}{\zeta_i^2}
 \right),
\end{align*}
and
\begin{align*}
\overline{\xi}_{\rm IS}^{(t)} & = \frac{1}{n}\sum_{i = 1}^n
\sqrt{
\sum_{k = 1}^d \left\lvert
\mathbb{E}_{\nu^{\otimes N}_i(\cdot\, ; \btheta)} \left[ \widehat{s}_{ik} \right]  - \frac{\partial}{\partial \theta_k} \log p_{\btheta}(\bY_i)
 \right\rvert^2
}
\\
& \leq \frac{1}{n} \sum_{i = 1}^n
\frac{2\sqrt{d}}{N\zeta_i}
 \left\lbrace
\frac{1}{\zeta_i} 4\lambda_i^2 \beta_i^{1/4}
+ \sqrt{\frac{12\lambda_i^4\sqrt{\beta}_i}{\zeta_i^2}
 \left(
1 + \frac{250\lambda_i}{\zeta_i} + \frac{9\lambda_i^2}{\zeta_i^2}
 \right)
}
 \right\rbrace
\\
& \leq \frac{4\sqrt{d}}{nN}
\sum_{i = 1}^n
\frac{\lambda_i^2\beta_i^{1/4}}{\zeta_i^2}
 \left\lbrace
2
+ \sqrt{3
 \left(
1 + \frac{250\lambda_i}{\zeta_i} + \frac{9\lambda_i^2}{\zeta_i^2}
 \right)
}
 \right\rbrace.
\end{align*}
\end{proof}
%------------------------------------------------------------------------------------------
%------------------------------------------------------------------------------------------

\sepline

%------------------------------------------------------------------------------------------
%------------------------------------------------------------------------------------------
\thmconvergence*
%------------------------------------------------------------------------------------------
%------------------------------------------------------------------------------------------
\begin{proof}
According to Proposition \ref{prop:smoothness}, the function $\ell$ is $L$-smooth. Using successively that the square function and the expectation are increasing on $\mathbb{R}_{+}$, for any real constant $\eta > 0$, Lemma \ref{lem:Moreau_mapping} yields
\begin{equation}
\label{eqn:proof-thm-1}
\mathbb{E} \left[
 \cp\lVert G_{\eta}^{(t)} \cp\rVert^2
 \right]
\leq \left( 1 + \frac{L\eta}{L\eta+1} \right)^2 \left( 1 + \sqrt{\frac{L\eta}{L\eta + 1}}  \right)^2
\mathbb{E} \left[
 \cp\lVert \nabla F_{\frac{\eta}{L\eta + 1}}\cp( \btheta^{(t)} \cp) \cp\rVert^2
 \right].
\end{equation}
To elaborate on the upper bound, we aim at using Lemma \ref{lem:control}.
Therefore, we first need to prove that we have control over the bias $\xi^{(t)}$ and the mean squared error $\sigma^{(t)}$ of the gradient estimates, as defined in Lemma \ref{lem:control}.

Let $t \in \mathbb{N}^*$. Due to the bilinearity of the scalar product, the following identity holds:
\begin{align*}
\sigma^{(t)} & =
\mathbb{E} \left[
\left\lVert
\widehat{g}^{(t)} + \nabla_{\btheta}\log p_{\btheta^{(t)}} \left( \bY_{i(t)} \right)
\right\rVert ^2
\mid \btheta^{(t)}
\right]
\\
& \hspace{1.5em}+\BB{
\mathbb{E} \left[
\left\lVert
\nabla_{\btheta}\ell( \btheta^{(t)} ) + \nabla_{\btheta}\log p_{\btheta^{(t)}} \left( \bY_{i(t)} \right)
\right\rVert ^2
\mid \btheta^{(t)}
\right]} \\
& \hspace{1.5em}
-2\mathbb{E} \left[
\left\langle
\widehat{g}^{(t)} + \nabla_{\btheta}\log p_{\btheta^{(t)}} \left( \bY_{i(t)} \right),
\nabla_{\btheta}\ell( \btheta^{(t)} ) + \nabla_{\btheta}\log p_{\btheta^{(t)}} \left( \bY_{i(t)} \right)
\right\rangle
\mid \btheta^{(t)}
\right]
\end{align*}
%As $\widehat{g}^{(t)} = -\widehat{g}^N_{i(t)}\big(\btheta^{(t)}\big)$,
Note that the first term of the right hand side corresponds to $\overline{\sigma}_{\rm IS}^{(t)}$.
Let bound the two remaining terms.
According to Proposition \ref{prop:score}, for all $i = 1, \ldots, n$, the function
\begin{equation*}
\btheta \mapsto \nabla_{\btheta}\ell \left( \btheta \right) + \nabla_{\btheta}\log p_{\btheta} \left( \bY_{i} \right)
\end{equation*}
is continuous on $\Theta$, and thereby bounded. Denote
\begin{equation*}
\Delta = \max_{i = 1, \ldots, n} \sup_{\btheta\in\Theta} \left\lVert
\nabla_{\btheta}\ell \left( \btheta \right) + \nabla_{\btheta}\log p_{\btheta} \left( \bY_{i} \right)
 \right\rVert.
\end{equation*}
We have
\begin{equation*}
\mathbb{E} \left[
 \cp\lVert
\nabla_{\btheta}\ell\cp( \btheta^{(t)} \cp) + \nabla_{\btheta}\log p_{\btheta^{(t)}} \left( \bY_{i(t)} \right)
 \cp\rVert ^2
\mid \btheta^{(t)}
 \right]
\leq \Delta^2
\end{equation*}
Moreover, $i(t)$ being a uniform random variable on $\{1, \ldots, n\}$
independent of the sequence $\btheta^{(t)}$,
\BB{
\begin{align*}
& \mathbb{E} \left[
 \cp\langle
 \widehat{g}^{(t)} + \nabla_{i(t)}^{(t)},
 \nabla_{\btheta}\ell\cp( \btheta^{(t)} \cp) + \nabla_{i(t)}^{(t)}
 \cp\rangle
\mid \btheta^{(t)}
 \right]
\\
& =
\mathbb{E} \left[ \mathbb{E} \left[
 \cp\langle
 -\widehat{s}^{N}_{i(t)}\cp( \btheta^{(t)} \cp) + \nabla_{i(t)}^{(t)},
 \nabla_{\btheta}\ell\cp( \btheta^{(t)} \cp) + \nabla_{i(t)}^{(t)}
 \cp\rangle
\mid \btheta^{(t)}, i(t)
 \right]
 \right]
\\
& =
\mathbb{E} \left[
 \cp\langle
-\mathbb{E} \left[ \widehat{s}^{N}_{i(t)}\cp( \btheta^{(t)} \cp) \mid \btheta^{(t)}, i(t)
\right] + \nabla_{i(t)}^{(t)},
\nabla_{\btheta}\ell\cp( \btheta^{(t)} \cp) + \nabla_{i(t)}^{(t)}
 \cp\rangle
 \right]
\\
& =
-\frac{1}{n}\sum_{i = 1}^n
 \cp\langle
 \mathbb{E}_{\nu_{i}^{\otimes N}(\cdot\,;\btheta^{(t)})} \left[ \widehat{s}^{N}_{i}\cp( \btheta^{(t)} \cp) \right] - \nabla_{i}^{(t)},
 \nabla_{\btheta}\ell\cp( \btheta^{(t)} \cp) + \nabla_{i}^{(t)}
 \cp\rangle,
\end{align*}}
\BB{with $\nabla^{(t)}_i = \nabla_{\btheta}\log p_{\btheta^{(t)}} \left( \bY_{i} \right)$ for $i = 1,\cdots,n$}.
Providing Assumption \ref{hyp:proposal} holds, the Cauchy-Schwarz inequality
along with Proposition \ref{prop:agapiou} yields
\BB{
\begin{align*}
& -2\mathbb{E} \left[
 \cp\langle
 \widehat{g}^{(t)} + \nabla_{i(t)}^{(t)},
 \nabla_{\btheta}\ell\cp( \btheta^{(t)} \cp) + \nabla_{i(t)}^{(t)}
 \cp\rangle
\mid \btheta^{(t)}
 \right]
\\
& \leq
\frac{2}{n}\sum_{i = 1}^n
 \cp\lVert
\mathbb{E}_{\nu_{i}^{\otimes N}(\cdot\,;\btheta^{(t)})} \left[
\widehat{s}^{N}_{i}\cp( \btheta^{(t)} \cp) \right] - \nabla_{i}^{(t)}
 \cp\rVert
 \cp\lVert
 \nabla_{\btheta}\ell\cp( \btheta^{(t)} \cp) + \nabla_{i}^{(t)}
 \cp\rVert
\\
& \leq \frac{2\Delta\sqrt{d}}{N} M_{\xi}.
\end{align*}}
In conclusion, for any $t\in\mathbb{N}^{*}$, we have a finite upper bound for the mean squared error $\sigma^{(t)}$, namely
\begin{equation*}
\sigma^{(t)} \leq \frac{d}{N}M_{\sigma}  + \Delta^2 + \frac{2\Delta\sqrt{d}}{N} M_{\xi}
= \Delta^2 + \frac{d}{N} \left( M_{\sigma} + \frac{2\Delta}{\sqrt{d}}M_{\xi} \right).
\end{equation*}
Conversely, leveraging  Assumption \ref{hyp:proposal} once again, Proposition \ref{prop:agapiou} shows that
we have a finite upper bound for the bias $\xi^{(t)}$. Indeed
\begin{align*}
\xi^{(t)} & = \left\lVert
\frac{1}{n}\sum_{i = 1}^n
\mathbb{E} \left[
-\widehat{s}_{i(t)}^{N}
\mid \btheta^{(t)}, i(t) = i
 \right]
+ \frac{1}{n}\sum_{i = 1}^n \nabla_{\btheta} \log p_{\btheta^{(t)}}(\bY_i)
 \right\rVert
\\
& =
 \left\lVert
\frac{1}{n}\sum_{i = 1}^n
\mathbb{E} \left[
-\widehat{s}_{i(t)}^{N} + \nabla_{\btheta} \log p_{\btheta^{(t)}}(\bY_{i(t)})
\mid \btheta^{(t)}, i(t) = i
 \right]
 \right\rVert
\\
& \leq
\frac{1}{n}\sum_{i = 1}^n
 \left\lVert
\mathbb{E} \left[
-\widehat{s}_{i(t)}^{N} + \nabla_{\btheta} \log p_{\btheta^{(t)}}(\bY_{i(t)})
\mid \btheta^{(t)}, i(t) = i
 \right]
 \right\rVert
= \overline{\xi}_{\rm IS}^{(t)} \leq \frac{\sqrt{d}}{N}M_{\xi}.
\end{align*}
Consequently, Lemma \ref{lem:control} applies for any real constant $\eta \in \big(0, 1/\max\{2\Gamma, L\}\big]$, as
\begin{equation*}
\frac{\eta}{L\eta + 1} \leq \frac{1}{L + \max\{2\Gamma, L\}} = \frac{1}{\max\{2\Gamma + L, 2L\}}.
\end{equation*}
After integrating both sides of Equation \eqref{eq:lyapunov_biased} with respect to $\btheta^{(t)}$ and plugging in the aforementionned upper bounds, we get
\begin{dmath*}
\frac{\gamma}{2} \mathbb{E} \left[ \left\lVert \nabla F_{\frac{\eta}{L\eta + 1}}\cp( \btheta^{(t)} \cp) \right\rVert^2 \right]
\leq
\mathbb{E} \left[ F_{\frac{\eta}{L\eta + 1}} \cp( \btheta^{(t)} \cp) \right]
- \mathbb{E} \left[ F_{\frac{\eta}{L\eta + 1}} \cp( \btheta^{(t+1)} \cp) \right]
+ \frac{(\gamma + \gamma^2\Gamma)(L\eta + 1)\sqrt{d}}{\eta N} M_{\xi}
+ \frac{\gamma^2(L\eta + 1)}{2\eta}
 \left\lbrace \Delta^2 + \frac{d}{N} \left( M_{\sigma} + \frac{2\Delta}{\sqrt{d}}M_{\xi} \right) + \Gamma^2 \right\rbrace.
\end{dmath*}
Then, summing along the iterations $t = 1, \ldots, T$ yields
\begin{dmath*}
\sum_{t = 1}^T \mathbb{E} \left[ \left\lVert \nabla F_{\frac{\eta}{L\eta + 1}}\cp( \btheta^{(t)} \cp) \right\rVert^2 \right]
\leq
\frac{2}{\gamma} \left\lbrace F_{\frac{\eta}{L\eta + 1}} \cp( \btheta^{(1)} \cp)
- \mathbb{E} \left[ F_{\frac{\eta}{L\eta + 1}} \cp( \btheta^{(T+1)} \cp) \right] \right\rbrace
+ \frac{2(L\eta + 1)T\sqrt{d}}{\eta N} M_{\xi}
+ \frac{\gamma(L\eta + 1)T}{\eta}
 \left[ \Delta^2 + \frac{d}{N} \left\lbrace M_{\sigma} + \frac{2(\Delta + \Gamma)}{\sqrt{d}}M_{\xi} \right\rbrace + \Gamma^2 \right].
\end{dmath*}
For $\btheta^{\mle} = \argmin_{\btheta\in\Theta}\ell(\btheta)$, according
to the definition of the Moreau envelope, for any $\btheta\in\Theta$
\begin{equation*}
\ell(\btheta^{\mle}) \leq F_{\sfrac{\eta}{(L\eta + 1)}} \left( \btheta \right),
\end{equation*}
and thus,
\begin{dmath*}
\frac{1}{T}\sum_{t = 1}^T \mathbb{E} \left[ \left\lVert \nabla F_{\frac{\eta}{L\eta + 1}}\cp( \btheta^{(t)} \cp) \right\rVert^2 \right]
\leq
\frac{2}{\gamma T} \left\lbrace F_{\frac{\eta}{L\eta + 1}} \cp( \btheta^{(1)} \cp)
- \ell(\btheta^{\mle}) \right\rbrace
+ \frac{2(L\eta + 1)\sqrt{d}}{\eta N} M_{\xi}
+ \frac{\gamma(L\eta + 1)}{\eta} \left[ \Delta^2 + \frac{d}{N} \left\lbrace M_{\sigma} + \frac{2(\Delta + \Gamma)}{\sqrt{d}}M_{\xi} \right\rbrace + \Gamma^2 \right].
\end{dmath*}
Combining the last inequality with Equation \eqref{eqn:proof-thm-1} and using $F_{\nicefrac{\eta}{(L\eta + 1)}} ( \btheta^{(1)} ) \leq \ell(\btheta^{(1)})$ since $\btheta^{(1)}\in \Theta$, we get for $\gamma = \gamma_0 / \sqrt{T}$
\begin{dmath*}
\frac{1}{T}\sum_{t = 1}^T \mathbb{E} \left[ \cp\lVert G_{\eta}^{(t)} \cp\rVert^2 \right]
\leq
\frac{2\tau}{\gamma_0 (L\eta + 1) \sqrt{T}} \left\lbrace F_{\frac{\eta}{L\eta + 1}} \cp( \btheta^{(1)} \cp)
- \ell(\btheta^{\mle}) \right\rbrace
+ \frac{2\tau\sqrt{d}}{\eta N} M_{\xi}
+ \frac{\gamma_0\tau}{\eta\sqrt{T}} \left[ \Delta^2 + \frac{d}{N} \left\lbrace M_{\sigma} + \frac{2(\Delta + \Gamma)}{\sqrt{d}}M_{\xi} \right\rbrace + \Gamma^2 \right].
\end{dmath*}
\end{proof}
%------------------------------------------------------------------------------------------
%------------------------------------------------------------------------------------------

\sepline

%------------------------------------------------------------------------------------------
%------------------------------------------------------------------------------------------

\section{Results on the importance sampling proposal distribution}
\label{sec:app-proposal}

%------------------------------------------------------------------------------------------
%------------------------------------------------------------------------------------------
\lemmamixt*
%------------------------------------------------------------------------------------------
%------------------------------------------------------------------------------------------
\begin{proof}
Let $i \in \{1, \ldots, n\}$, $\alpha\in(0, 1]$ and $\delta > 1$.

\paragraph*{(A1.1)} Each component of a mixture distribution being a non-negative function, we have for any $\bV\in\mathbb{R}^q$ and any $\btheta\in\mathbb{R}^d$,
\begin{equation*}
\rho_{\btheta, i}(\bV) =  \frac{p_{\btheta}(\bY_i, \bV)}{\mathcal{GM}(\bV; \bmu_i(\btheta), \bS_i(\btheta), \alpha, \delta)}
\leq \frac{1}{\alpha} \frac{p_{\btheta}(\bY_i, \bV)}{\mathcal{N}(\bV; \bmu_i(\btheta), \delta\mathbf{I}_q)}.
\end{equation*}
On the compact set $\Theta$, there exist real constants $K_i^{\Theta} > 0$ and $\kappa_i^{\Theta}$ (Lemma \ref{lem:unif-bound}) such that
\begin{align*}
\log \frac{p_{\btheta}(\bY_i, \bV)}{\mathcal{N}(\bV; \bmu_i(\btheta), \delta\mathbf{I}_q)}
& \leq
K^\Theta_i \lVert \bV \rVert + \kappa_i^\Theta - \frac{1}{2} \lVert \bV \rVert^2
+ \frac{1}{2\delta} \left\lVert \bV - \bmu_i(\btheta) \right\rVert^2 + \frac{q}{2}\log(\delta).
%\\
%& \leq
% K_i^\Theta \lVert \bV \rVert + \kappa_i^\Theta - \frac{\delta - 1}{2\delta} \lVert \bV \rVert^2
%+ \frac{1}{2\delta} \left\lVert \bmu_i(\btheta) \right\rVert^2 - \frac{1}{\delta} \langle \bmu_i(\btheta), \bV\rangle + \frac{q}{2}\log(\delta).
\end{align*}
Using the Cauchy-Schwarz inequality, we get
\begin{equation*}
\left\lVert \bV - \bmu_i(\btheta) \right\rVert^2
= \left\lVert \bV \right\rVert^2 + \left\lVert \bmu_i(\btheta) \right\rVert^2 - 2\langle \bmu_i(\btheta), \bV\rangle
\leq  \left\lVert \bV \right\rVert^2 + \left\lVert \bmu_i(\btheta) \right\rVert^2 + 2 \lVert \bV \rVert \lVert \bmu_i(\btheta) \rVert.
\end{equation*}
$\btheta \mapsto \bmu_i(\btheta)$ is continuous on the compact set $\Theta$, and hence bounded, say by $u_\star$. It follows that
\begin{equation*}
\left\lVert \bV - \bmu_i(\btheta) \right\rVert^2
\leq  \left\lVert \bV \right\rVert^2 + u_\star^2 + 2 u_\star \lVert \bV \rVert .
\end{equation*}
Consequently, we can derive an upper bound independent of $\btheta$, namely
\begin{equation*}
\frac{p_{\btheta}(\bY_i, \bV)}{\mathcal{N}(\bV; \bmu_i(\btheta), \delta\mathbf{I}_q)}
\leq
\exp \left[
 \left( K_i^\Theta + \frac{u_\star}{\delta} \right) \lVert \bV \rVert
 - \frac{\delta - 1}{2\delta} \lVert \bV \rVert^2
 + \kappa_i^\Theta
+ \frac{u_\star^2}{2\delta} + \frac{q}{2}\log(\delta)
 \right].
\end{equation*}
Providing $\delta > 1$, the quadratic term in $\bV$ overweight any other terms and the supremum in $\bV$ is finite:
\begin{equation*}
\sup_{(\btheta, \bV) \in \Theta\times\mathbb{R}^q}
\rho_{\btheta, i} (\bV)
\leq
\frac{1}{\alpha} \sup_{(\btheta, \bV) \in \Theta\times\mathbb{R}^q}  \frac{p_{\btheta}(\bY_i, \bV)}{\mathcal{N}(\bV; \bmu_i(\btheta), \delta\mathbf{I}_q)} < \infty.
\end{equation*}
\paragraph*{(A1.2)} Following Equation \eqref{app:eqn-partial}, there is a unique $j$ in $\{1, \ldots, p\}$ such that
\begin{equation*}
 \left\lvert
\frac{\partial}{\partial \theta_r} \log p_{\btheta}(\bY_i, \bW_i)
 \right\rvert
\leq
 \left( \lVert \bx_i \rVert + \lVert \bW_i \rVert \right)
 \left\lbrace \lVert \bY_{i} \rVert + \exp(Z_{ij}) \right\rbrace.
\end{equation*}
We then have,
\begin{align*}
 \left\lvert
\frac{\partial}{\partial \theta_r} \log p_{\btheta}(\bY_i, \bW_i)
 \right\rvert^4
& \leq
\left(8\lVert \bx_i \rVert^4 + 8\lVert \bW_i \rVert^4 \right)
\left\lbrace 8\lVert \bY_{i} \rVert^4 + 8\exp(4Z_{ij}) \right\rbrace.
\\
& \leq
64 \lVert \bY_{i} \rVert^4 \left(
\lVert \bx_i \rVert^4
+ \lVert \bW_i \rVert^4 \right)
\\
& \hspace{1.5em}
+ \BB{64\exp(4 (\bB_{j}^\top\bx_{i} + o_{ij} + \bC_j^\top\bW_i)) \left\lbrace \lVert \bx_i \rVert^4
+ \lVert \bW_i \rVert^4
\right\rbrace,}
\end{align*}
where $\bB_j = (B_{1j}, \ldots, B_{pj})^\top$ and $\bC_j = (C_{j1}, \ldots, C_{jq})^\top$ stand for the $j$-th column and row of $\bB$ and $\bC$, respectively.

Let show that the upper bound admits a finite expectation with respect to a multivariate Gaussian distribution $\mathcal{N}(\mu, S)$, for any $\mu\in\mathbb{R}^q$ and $S\in\mathbb{S}^q_{++}$. To demonstrate this, we use the following identities. For any $\bw \in\mathbb{R}^q$, a straightforward rewriting yields
\begin{equation}
\label{app:eqn-id-gauss}
\mathcal{N}(\bw;\mu, S) \exp\left( 4 \bC_j^\top\bw \right)
= \mathcal{N}(\bw; \mu + 4 S\bC_j, S)\exp\left( 4\bC_j^\top \mu + 8 \bC_j^\top S \bC_j \right).
\end{equation}
Moreover, we have (see for instance \cite{simon2002probability})
\begin{equation*}
\mathbb{E}_{\mathcal{N}(\mathbf{0}_q, \mathbf{I}_q)}\left[
\left\lVert \bW_i \right\rVert^{4}
\right] \leq q(q + 2).
\end{equation*}
%Denote $L \in \mathcal{M}_{q\times q}(\mathbb{R})$ the lower triangular matrix given by the Cholesky decomposition of $S$, that is $S = L L^\top$.
Using that
\begin{align*}
\mathbb{E}_{\mathcal{N}(\mu, S)}\left[ \left\lVert \bW_i \right\rVert^4 \right]
& = \mathbb{E}_{\mathcal{N}(\mathbf{0}_q, \mathbf{I}_q)}\left[ \cp\lVert \mu + S^{\nicefrac{1}{2}}\bW_i \cp\rVert^4 \right] \\
& \leq 8 \left\lVert \mu \right\rVert^4 + 8 \cp\lVert S^{\nicefrac{1}{2}} \cp\rVert^4 \mathbb{E}_{\mathcal{N}(\mathbf{0}_q, \mathbf{I}_q)}\left[ \left\lVert \bW_i \right\rVert^4 \right],
\end{align*}
we get
\begin{equation}
\label{app:eqn-id-moment}
\mathbb{E}_{\mathcal{N}(\mu, S)}\left[ \left\lVert \bW_i \right\rVert^4 \right]
\leq 8 \left\lVert \mu \right\rVert^4 + 8q(q + 2) \cp\lVert S^{\nicefrac{1}{2}} \cp\rVert^4
\end{equation}
Equations \eqref{app:eqn-id-gauss} and \eqref{app:eqn-id-moment} thus lead to
\begin{align*}
\Phi_j(\btheta, \mu, S)
& = \exp\left( 4\bC_j^\top \mu + 8 \bC_j^\top S \bC_j \right)
\\
& \BB{= \mathbb{E}_{\mathcal{N}(\mu, S)}\left[
\exp\left( 4 \bC_j^\top\bW_i \right)
\right]},
\\
\Psi_j(\btheta, \mu, S)
& = \exp\left( 4\bC_j^\top \mu + 8 \bC_j^\top S \bC_j \right)
\left\lbrace 8 \left\lVert  \mu + 4 S\bC_j \right\rVert^4 + 8q(q + 2) \cp\lVert S^{\nicefrac{1}{2}} \cp\rVert^4  \right\rbrace
\\
& \geq \mathbb{E}_{\mathcal{N}(\mu, S)}\left[
\left\lVert \bW_i \right\rVert^4 \exp\left( 4 \bC_j^\top\bW_i \right)
\right].
\end{align*}
Therefore, for any $\mu\in\mathbb{R}^q$ and $S\in\mathbb{S}^q_{++}$
\begin{align*}
\mathbb{E}_{\mathcal{N}(\mu, S)}\left[
\left\lvert
\frac{\partial}{\partial \theta_r} \log p_{\btheta}(\bY_i, \bW_i)
\right\rvert^4
\right]
& \leq
64 \lVert \bY_{i} \rVert^4 \Bigg\lbrace
\lVert \bx_i \rVert^4
+8 \left\lVert \mu \right\rVert^4 \\
& \hspace{4.5em}\BB{ + 8q(q + 2) \cp\lVert S^{\nicefrac{1}{2}} \cp\rVert^4} \Bigg\rbrace
\\
&  \hspace{1.5em}
+ 64\exp(4 \bB_{j}^\top\bx_{i} + 4 o_{ij}) \lVert \bx_i \rVert^4 \Phi_j(\btheta, \mu, S)
\\
&  \hspace{1.5em}
+ 64\exp(4 \bB_{j}^\top\bx_{i} + 4 o_{ij}) \Psi_j(\btheta, \mu, S).
\end{align*}
If $\btheta \mapsto \bmu_i(\btheta)$ and $\btheta \mapsto \bS_i(\btheta)$ are
continuous on $\Theta$, the functions
\begin{align*}
\btheta & \mapsto  \bS_i(\btheta)^{\nicefrac{1}{2}},
\\
\btheta & \mapsto (1 - \alpha)\Phi_j(\btheta, \bmu_i(\btheta), \bS_i(\btheta)) + \alpha\Phi_j(\btheta, \bmu_i(\btheta), \delta\mathbf{I}_q),
\\
\btheta & \mapsto (1 - \alpha)\Psi_j(\btheta, \bmu_i(\btheta), \bS_i(\btheta)) + \alpha\Psi_j(\btheta, \bmu_i(\btheta), \delta\mathbf{I}_q),
\end{align*}
are also continuous on the compact set $\Theta$ (by a composition argument), and therefore bounded on $\Theta$, say by $c_\star$, $\phi_\star$ and $\psi_\star$ respectively.
On the other hand,
\begin{equation*}
\exp(4 \bB_{j}^\top\bx_{i} + 4 o_{ij})
\leq \exp(4 \lVert \bx_{i} \rVert \lVert \bB_{j}\rVert + 4 \lVert \bo_{i} \rVert)
\leq \exp(4 \lVert \bx_{i} \rVert \sup_{\btheta\in\Theta} \lVert \btheta \rVert + 4 \lVert \bo_{i} \rVert).
\end{equation*}
Consequently, with $u_\star$ the bound of $\btheta \mapsto \bmu_i(\btheta)$ on $\Theta$, for any $r\in\{1, \ldots, d\}$,
\begin{align*}
\mathbb{E}_{\nu_i(\cdot\,;\btheta)}\left[
\left\lvert
\frac{\partial}{\partial \theta_r} \log p_{\btheta}(\bY_i, \bW_i)
\right\rvert^4
\right]
& \leq
64 \lVert \bY_{i} \rVert^4
\Bigg[
\lVert \bx_i \rVert^4 + 8u_\star \\
& \hspace{1.5em} + \BB{8 q(q + 2) \left\lbrace (1-\alpha)c_\star^4 + \alpha\delta^4\right\rbrace}
\Bigg]
\\
&
+\BB{ 64\exp(4 \lVert \bx_{i} \rVert \sup_{\btheta\in\Theta} \lVert \btheta \rVert + 4 \lVert \bo_{i} \rVert)
\left( \lVert \bx_i \rVert^4 \phi_\star + \psi_\star \right)},
\end{align*}
which yields the conclusion.
%In conclusion,
%\begin{equation*}
%\sup_{\btheta\in\Theta} \mathbb{E}_{\nu_i(\cdot\,;\btheta)}\left[
%\left\lVert
%\nabla_{\btheta} \log p_{\btheta}(\bY_i, \bW_i)
%\right\rVert^4
%\right] < \infty.
%\end{equation*}
\end{proof}
%------------------------------------------------------------------------------------------
%------------------------------------------------------------------------------------------

%------------------------------------------------------------------------------------------
%------------------------------------------------------------------------------------------
\begin{lemma}
\label{lemma:continuity_moment}
When considering Model \eqref{eqn:PLNPCA}, for any individual $i = 1, \ldots, n$, and any $k\in\mathbb{N}^*$, the function
\begin{equation*}
\btheta \mapsto \int_{\mathbb{R}^q} \bw^k p_{\btheta}(\mathrm{d}\bw \mid \bY_i)
\end{equation*}
is continuous on $\mathbb R^d$.
\end{lemma}
%------------------------------------------------------------------------------------------
%------------------------------------------------------------------------------------------
\begin{proof}
Let $i \in \{1, \ldots, n\}$, $k\in\mathbb{N}^*$, and $\Theta \subset \mathbb{R}^d$ a non-empty bounded and open set.
For any $\bw\in\mathbb{R}^q$ and $\btheta\in\Theta$, the Bayes rule yields
\begin{equation*}
p_{\btheta}(\bw \mid \bY_i) = \frac{p_{\btheta}(\bY_i \mid \bw_i)\mathcal{N} \left(\bw;\boldsymbol{0}_q, \boldsymbol{I}_q \right)}{p_{\btheta}(\bY_i)}.
\end{equation*}
Given $\bw\in\mathbb{R}^q$, $\btheta \mapsto p_{\btheta}(\bY_i \mid \bW)$ is continuous on $\Theta$, and $\btheta \mapsto p_{\btheta}(\bY_i)$ is
continuous and positive on $\Theta$ (see Proposition \ref{prop:score} and its proof). This proves the continuity of $\btheta \mapsto \bw^k p_{\btheta}(\bw \mid \bY_i)$ on $\Theta$.

Additionally, Lemma \ref{lem:unif-bound} and Proposition \ref{prop:smoothness} state that there are real constants $K_i^{\Theta} > 0$, $\kappa_i^{\Theta}$, and $\zeta_i > 0$, such that for any $\btheta\in\Theta$
\begin{equation*}
\left\lVert \bw^k {p_{\btheta}(\bw \mid \bY_i)} \right\rVert
\leq
\frac{\lVert \bw \rVert^k}{\zeta_i}
\exp \left\{
K_i^{\Theta} \lVert \bw \rVert - \frac{1}{2} \lVert \bw \rVert^2 + \kappa_i^\Theta - \frac{q}{2}\log(2\pi)
 \right\}.
\end{equation*}
The upper bound is Lebesgue integrable on $\mathbb{R}^q$ and independent of
$\btheta\in\Theta$. Consequently, the dominated convergence theorem yields the continuity on any non-empty bounded and open set.

\end{proof}

%------------------------------------------------------------------------------------------
%------------------------------------------------------------------------------------------

% \section{Simulation study}
% \label{sec:app-experiments}

\end{appendix}

\end{document}